\documentclass[reqno,11pt,letterpaper]{amsart}
\usepackage{amsmath,amssymb,amsthm,graphicx,mathrsfs,url}
\usepackage[usenames,dvipsnames]{color}
\usepackage[colorlinks=true,linkcolor=Red,citecolor=Green]{hyperref}
\usepackage{amsxtra}

\usepackage{bbm}
\usepackage{epstopdf}
\usepackage[dvipsnames]{xcolor}

\usepackage{verbatim}
\usepackage{graphicx}

\def\?[#1]{\textbf{[#1]}\marginpar{\Large{\textbf{??}}}}

\numberwithin{equation}{section}

\let\epsilon=\varepsilon % sorry Knuth
\newcommand{\dd}{\, {\rm d}}

\allowdisplaybreaks
\usepackage[%top=1.25in, bottom=1.25in,
left=0.65in, right=0.65in]{geometry}
\newcommand{\RR}{{\mathbb R}}

\newcommand{\R}{\mathbb{R}}

\usepackage{amsxtra}
\usepackage{epstopdf}
\newtheorem{theorem}{Theorem}
\newtheorem{proposition}{Proposition}[section]

\newtheorem{lemma}[proposition]{Lemma}
\newtheorem{corollary}[proposition]{Corollary}
\newtheorem{remark}{Remark}

\title[Well-Posedness near Rayleigh-Jeans Equilibria on the whole space]{Global Well-Posedness near Rayleigh-Jeans Equilibria for the Cubic NLS Wave Kinetic Equation}

\author{Miguel Escobedo}
\address{Departamento de Matem\'aticas, Euskal Herriko Unibertsitatea, Apartado 644, E--48080 Bilbao, Spain}
\email{miguel.escobedo@ehu.es}

\author{Angeliki Menegaki} 
\address{Department of Mathematics, Huxley building, South Kensington campus, Imperial College London, London SW7 2AZ, United Kingdom}
\email{a.menegaki@imperial.ac.uk}

\begin{document}

\maketitle

\begin{abstract} 
We study the dynamics of the kinetic wave equation associated to the three dimensional Schr\"{o}dinger equation close to Rayleigh-Jeans equilibria. We first prove that the linearised operator generates a semigroup of contractions in $L^2((0,\infty);\sqrt \omega \dd\omega )$. Considering the family of nonsingular Rayleigh-Jeans spectra, we prove that the linearised operator possesses a spectral gap, despite the non-compactness of the integral collisional operator,  and thus obtain an exponential relaxation for the linear semigroup. 
We then prove bilinear and trilinear estimates in the relevant norm for the nonlinear terms and deduce global well-posedness and exponential relaxation for sufficiently small relative perturbations of Rayleigh-Jeans equilibria. To our knowledge, this is the first global strong well-posedness and asymptotic stability result near a nonzero thermodynamic equilibrium for the full spatially homogeneous
four-wave kinetic equation associated with the cubic
nonlinear Schr\"odinger equation.

%Then the full non-linear Boltzmann equation is solved perturbatively for initial data close to Rayleigh-Jeans equilibrium. Global in time well-posedness is thus deduced in $L^2((0,\infty);\sqrt \omega \dd\omega )$, after controlling the nonlinearities in the whole frequency space in the relevant norm.
%We study the dynamics of the kinetic wave equation associated to the 3d Schr\"{o}dinger equation \textcolor{blue}{close to  Rayleigh Jeans equiibria}. We first study the linearised operator around non-singular Rayleigh-Jeans equilibria, and we show existence of a spectral gap and thus exponential relaxation of the linear semigroup in the natural-for-dissipation functional space $L^2(\sqrt{\omega})$. 
%Then the full non-linear Boltzmann equation  is solved for
%initial data close to equilibrium. The non-linearity is treated as a perturbation
%of the linear problem. Global in time well-posedness in $L^2(\sqrt{\omega})$ is proved, after controlling the nonlinearities in the whole frequency space.
\end{abstract}

\tableofcontents

\section{Introduction}

In this article we establish the nonlinear asymptotic stability of the nonsingular Rayleigh-Jeans equilibria for the spatially homogeneous and isotropic four-wave kinetic equation associated with the cubic nonlinear Schr\"odinger equation on $\mathbb{R}^3$. 
Our analysis is performed on the
whole frequency domain, in particular without introducing an ultraviolet cutoff and yields the first global strong well-posedness and asymptotic stability result near thermodynamic equilibria for the spatially homogeneous wave kinetic equation associated with the cubic nonlinear Schr\"odinger equation. 
This article thus contributes to the broader programme in wave turbulence theory/wave kinetic theory concerned with understanding the global-in-time dynamics of wave kinetic equations and, in particular, with determining which classes of solutions remain globally regular and which may develop singular behaviour or condensation. 
This question is also emphasized by Deng and Hani in \cite[Section 1.5.3]{DengHani3}, where they ask whether condensation should be expected generically or only for particular classes of solutions. We identify a rigorous perturbative regime in which condensation does not occur: sufficiently small perturbations of a nonsingular Rayleigh-Jeans equilibrium remain global and relax exponentially to equilibrium.
%Thus, at least in the weighted $L^2$-topology considered here, condensation is not a universal feature of the dynamics.
%In this article we study the evolution of perturbations in an $L^2$ setting, of non-singular Rayleigh Jeans  spectra of the wave kinetic equation for Schr\"odinger waves on the whole space $\R^3$. 

\vspace{-0.5cm}

\subsection{Background of Wave Turbulence Theory}  
Wave turbulence theory is the statistical theory of large systems of weakly nonlinear wave systems generally evolving away from thermodynamic equilibrium. 
The theory aims to understand in a statistical way how the energy and the mass is then distributed in such random wave systems through the resonant interactions. The central equation describing the average distribution of energy is the kinetic wave equation. 

Early kinetic descriptions of wave systems appeared in 1929 in Peierls' theory of phonons, that is vibrations in weakly anharmonic crystals describing heat conduction, see \cite{Peierls1929}, but also in 1962 in Hasselmann's theory of nonlinear energy transfer for surface gravity waves, see \cite{Hasselmann:1962, Hasselmann1963}. The modern energy/mass  cascade picture was developed by Zakharov, L'vov and Falkovich \cite{ZakLvovFalkov}. The research area of wave turbulence theory thus is attracting more and more  attention due to its applications to a broad range of physical systems, including ocean waves,  gravity waves, phonons, plasma waves and nonlinear optics \cite{Dyachenko:1992}, elastic plates \cite{During:2017} but also Bose-Einstein condensates. We refer to the books \cite{Nazarenko, NewellRumpf, Zakharov98} for more details.

At the mathematical level, this formal theory has recently gained important attention with a lot of progress building its rigorous foundation. For the cubic nonlinear Schr\"odinger equation, in a
sequence of works, the homogeneous wave kinetic equation was rigorously derived from random weakly nonlinear Schr\"odinger dynamics, \cite{DengHani1,DengHani1/2, DengHani2, DengHani3} first up to small multiples of the kinetic timescale and subsequently
over arbitrarily long kinetic intervals for as long as sufficiently regular solutions of the kinetic equation exist. 
Beyond the cubic nonlinear Schr\"odinger equation, significant mathematical progress towards a wave turbulence theory for the quasilinear two-dimensional gravity water wave system is obtained in \cite{DengIonescuPusateriI,DengIonescuPusateriII}. Moreover, rigorous kinetic limits have been investigated for one-dimensional systems such as weakly anharmonic oscillator chains, see \cite{VassilevWu2026} for the FPUT system and \cite{Vassilev2025} for the MMT model. We refer also to \cite{StaffilaniTran2021} for the stochastic Zakharov-Kuznetsov dynamics, and also to  \cite{deSuzzoniStingoTouati2025} for a  coupled Klein-Gordon system. 

All  these results are conditional on the wave kinetic equation being well-posed, and thus they  make the development of a global solution theory
for the wave kinetic equation a natural and important problem. 

\subsection{The homogeneous kinetic wave equation} 
We denote by $\textbf{k} \in \mathbb{R}^3$ the frequency vector and by $\omega(\textbf{k}) = |\textbf{k}|^2$ the dispersion relation of the cubic nonlinear Schr\"odinger equation.

We are working under the isotropy assumption, where the density of waves function, $f$,  depends on
$\textbf{k}$ only through its energy:
$\widetilde f(t,\mathbf{k}) = f(t,|\mathbf{k}|^2).$ 

After absorbing a harmless positive constant into the time
variable, the isotropic and spatial homogeneous  wave kinetic equation reads 
\begin{equation}\label{eq:NL evolution1}
\begin{split}
    &\partial_t f(t,\omega) = \mathcal{C}(f(t))(\omega )\equiv -
     \iint_{D(\omega)}\Omega (\omega , \omega _1, \omega _2, \omega _3) q(f)\dd\omega _1 \dd\omega _2 \quad \text{ where} \\ 
&\Omega (\omega , \omega _1, \omega _2, \omega _3)= \frac{\text{min}(\sqrt{\omega},\sqrt{\omega_1},\sqrt{\omega_2},\sqrt{\omega_1+\omega_2-\omega})}{\sqrt{\omega}}, \\
    & q(f) \equiv q(f)(\omega , \omega _1, \omega _2, \omega _3) =  f f_2 f_3 + f f_1 f_3 - f_1 f_2 f - f_1 f_2 f_3,   \\
  &D(\omega ) = \left\{(\omega _1, \omega _2); \omega _1\ge 0,\,\omega _2\ge 0,\,\omega _1+\omega _2\ge \omega  \right\},
 \end{split}
\end{equation}
for $f(t, \omega)$ the  density of waves with energy $\omega$,  
with the usual notation $f=f(t, \omega )$, $f_i=f(t, \omega _i)$ for $i=1, 2, 3$ and    $\omega _3:= \omega_1+ \omega_2-\omega$.

%\textcolor{blue}{The equation may also be written using the frequency density of the waves: 
%\begin{equation}
%\label{eqfreq}
%\begin{split}
%&k=\omega ^{1/2}\\
%&\widetilde f(t, k)=f(t,\omega)
%\end{split}
%\end{equation}
%and reads   
%\begin{equation}\label{eq:NL evolution1b}
%\begin{split}
 %   &\partial_t \widetilde f(t,k) =\widetilde Q \left(\widetilde f(t) \right)(k)= -
  %   \iint_{\widetilde D(\omega)}\widetilde \Omega (k, k _1, k_2, k_3) q(\widetilde f) k_1 k_2 \dd k _1 \dd k_2 \quad  \text{ where} \\ 
%&\widetilde \Omega (k , k_1, k _2, k _3)=      \frac{\text{min}( k , k_1, k _2, k _3)}{k}, \\
  %   &q(\widetilde f) \equiv q(\widetilde f)(\omega , \omega _1, \omega _2, \omega _3) =  \widetilde f \widetilde f_2 \widetilde f_3 + \widetilde f \widetilde f_1 \widetilde f_3 - \widetilde f_1 \widetilde f_2 \widetilde f - \widetilde  f_1 \widetilde  f_2\widetilde  f_3,  \\
%  &\widetilde D(k) =\left\{( k _1, k _2)\in \RR_+^2;  k _1^2+k _2^2\ge k^2  \right\}.
% \end{split}
%\end{equation}}

The formal thermodynamic equilibria of the equation are the Rayleigh-Jeans distributions. The general form of Rayleigh-Jeans (RJ) spectra, in the $\omega$ variables, is
\begin{align*}
 \forall T \geq 0,\,\,\forall \mu_{\text{chem}} \le 0:\,\,\,\,\, n_{T, \mu_{\text{chem}}}(\omega) = \frac {T} {\omega -\mu_{\text{chem}} },\,\ \omega \ge 0,
\end{align*}

The nonsingular RJ distributions correspond to the case $\mu_{\text{chem}}<0$, whereas  $\mu_{\text{chem}}=0$ yields the singular RJ profile. Also $T=0$ yields the trivial one. By analogy with the gas of particles the two parameters $T$ and $\mu_{\text{chem}}$  are interpreted in the literature of physics as the temperature and chemical potential respectively of the system of waves. 
By homogeneity of the integral term of the equation, the temperature parameter $T$ of the RJ distributions may be absorbed via a change of time variable in the equation. Therefore,
%in order to simplify as much as possible our notation, and 
without any loss of generality, we  consider $T\equiv 1$, we set $\mu:=-\mu_{\text{chem}}$, and denote the family of RJ functions that we work with as 
$$  n_\mu (\omega )=\frac{1}{\omega +\mu  },\,\,\,  \mu \ge 0.$$

For every $\mu \ge 0$ the function $q(n_\mu)$ (the cubic nonlinearities) is identically zero and therefore, the RJ spectra are stationary solutions of \eqref{eq:NL evolution1}, in particular RJ profiles are the formal detailed-balance equilibria. 
Equation (\ref{eq:NL evolution1}) has additional stationary solutions called  Kolmogorov-Zakharov spectra, which are non-equilibrium solutions carrying nonzero wave-action or energy fluxes \cite{Dyachenko:1992}. Both types of solutions have been considered in the literature on wave turbulence (cf. \cite{Dyachenko:1992, ZakLvovFalkov}). The RJ spectra, contrary to the Kolmogorov-Zakharov spectra, have zero wave action and energy fluxes, describing a system of waves at equilibrium \cite{Dyachenko:1992}.

%The case $\mu =0$ is different due to its singularity at $\omega =0$ as was studied in our previous work  \cite{EscMenegaki26}. There, truncations of this profile generate a Dirac mass at zero frequency, while the corresponding linearised dynamics
%concentrates towards the origin in infinite time. Thus the instability mechanism of the singular Rayleigh-Jeans state found in \cite{EscMenegaki26} does not extend to the
%nonsingular family, in the perturbative topology considered here.

\subsection{The setting} Since in all our arguments the parameter $\mu \ge 0$ will be arbitrary but fixed, we further simplify our notation and denote the RJ distributions simply as $n$, and $n(\omega _i)$ as $n_i$ for $i=1, 2, 3$. It will always be clear by the context if $\mu >0$ or $\mu =0$. 

We consider for the moment any RJ equilibrium $n$. Then for solutions of the form
\begin{equation} \label{eq:perturbat_setting}
    f(t,\omega)= n (\omega) (1+ g(t,\omega)), 
\end{equation}
we look at the evolution of the perturbation $g$. The equation \eqref{eq:NL evolution1} then reads
\begin{equation}\label{eq:NL evolution}
\begin{split}
    \partial_t g(t,\omega) &= -n^{-1}(\omega)
     \iint_{D(\omega)}\Omega (\omega , \omega _1, \omega_2, \omega _3) q(f) \dd \omega_1 \dd \omega_2   = L_\mu g + \Gamma_\mu(g,g) + Q_\mu(g,g,g)
 \end{split}
\end{equation}
 where $L_\mu $ is the linear operator
\begin{equation} \label{eq: A_mu+K_mu}
\begin{split} 
L_\mu g & = \iint_{D(\omega)} \Omega (\omega , \omega _1, \omega _2, \omega _3) n_1n_2n_3 \Big[-\frac{g_0}{n}  - \frac{g_3}{n_3} + \frac{g_1}{n_1} + \frac{g_2}{n_2}   
\Big] \dd \omega_1 \dd \omega_2 \\
&=: -[A_\mu g](\omega) + [K_\mu g](\omega)
\end{split} 
\end{equation} 
where again $n=n(\omega )$, $n_i=n(\omega_i), i=1, 2, 3$, and 
where $A_\mu$ is the multiplication part of the operator, while $K_\mu$ is the integral operator containing the remaining three terms. The term  $ \Gamma_\mu (g,g)$ contains the quadratic nonlinearities
%: 
%\begin{equation}\label{eq: def of Gamma}
%\begin{split} 
%\Gamma_\mu(g,g)& = \iint_{D(\omega)} \Omega (\omega , \omega _1, \omega _2, \omega _3)
%n_1n_2n_3 
%\times \\ &\times 
%\Big[ -g_0g_3 ( n_2^{-1} + n_1^{-1} ) + 
 %   g_0g_1 ( n_3^{-1} - n_2^{-1}) + g_0g_2 (n_3^{-1} - n_1^{-1}) \\ 
 %   &\hspace{1cm} + 
  %  g_3g_1 (n_1^{-1} - n_3^{-1}) + g_3g_2 (n_2^{-1} - n_3^{-1}) + g_1g_2 (n_2^{-1} + n_1^{-1}  )
%\Big] \dd \omega_1 \dd \omega_2, 
%\end{split}
%\end{equation} 
 and $Q_\mu (g,g,g)$ contains the cubic nonlinearities.
 %: 
%\begin{equation}\label{eq: def of Q}
%\begin{split} 
 %Q_\mu(g,g,g)
 %= 
 %&\iint_{D(\omega)} \Omega (\omega , \omega _1, \omega _2, \omega _3)   n_1n_2n_3\times \\
 % &\hspace{0.7cm}\times \Big[   n_3^{-1} g_0g_1g_2  +  n_0^{-1}g_1g_2g_3 -  n_1^{-1}g_0g_2g_3  -  n_2^{-1}g_0g_1g_3 \Big] \dd \omega_1 \dd \omega_2.
 %\end{split}
%\end{equation}  
Their explicit expressions are given in Section \ref{sec:LWP}, Eq. \eqref{eq: def of Gamma}, \eqref{eq: def of Q}, respectively.

\subsection{Main results}
We denote by $\Pi_\mu$ (or simply by $\Pi$) the orthogonal projection onto $\operatorname{Ker}(L_\mu)$. We start with a proposition determining the collisional invariants (the $\operatorname{Ker}(L_\mu)$) in $L^2(\sqrt{\omega}\dd\omega)$. 
 
\begin{proposition}[Kernel of the linearised operator] \label{Prop:Kernel_L}
    Let $\widetilde n_\mu(\textbf{k})=\frac1{|\textbf{k}|^2+\mu}, \mu > 0,$ and $\tilde g \in L^2(\mathbb{R}^3)$. Then $\widetilde g \in \operatorname{Ker}(\widetilde L_\mu)$ if and only if $\Phi(\textbf{k}) := \frac{ \widetilde g(\textbf{k})}{\widetilde n_\mu (\textbf{k})}$ is a collisional invariant,
    namely
$$ \Phi(\textbf{k}_1)+\Phi(\textbf{k}_2) = \Phi(\textbf{k})+\Phi(\textbf{k}_3) $$ for almost every quadruple satisfying
$$
\textbf{k}_1+\textbf{k}_2 = \textbf{k}+\textbf{k}_3, \quad  |\textbf{k}_1|^2+|\textbf{k}_2|^2 =
    |\textbf{k}|^2+|\textbf{k}_3|^2, 
$$
if and only if  
    $$
    \widetilde g(\textbf{k})= \widetilde n_\mu(\textbf{k}) ( \alpha  + \beta \cdot \textbf{k} + \gamma |\textbf{k}|^2 ),\ \qquad \ \text{ for some } \alpha, \gamma \in \mathbb{R},\ \beta \in \mathbb{R}^3,$$ 
    for almost every $\textbf{k}\in \mathbb{R}^3$,  with the corresponding right-hand side belonging in $L^2(\mathbb{R}^3)$. 

    Among these collisional invariants only  $\widetilde n_\mu \in L^2(\mathbb{R}^3)$ and so it holds that $$\operatorname{Ker}_{L^2(\mathbb{R}^3)}( \widetilde L_\mu) = \langle \widetilde n_\mu \rangle.$$ 
    
    It follows that, in radial energy variables, $L^2((0,\infty), \sqrt{\omega} \dd\omega)$, $g$ belongs to $\operatorname{Ker}_{L^2(\sqrt{\omega})}( L_\mu)$ if and only if  $g(\omega )=a n(\omega )$ for some $a\in \RR$.
\end{proposition}

We gather our findings on the linear operator $L_\mu$ in the following theorem:

\begin{theorem}
\label{MainTh2}
For every $\mu \ge 0$ the operator $-L_\mu $ defined in \eqref{eq: A_mu+K_mu} satisfies the following properties 
\begin{itemize}
    \item[(i)] is symmetric, nonnegative on $L^2(\sqrt \omega \dd\omega )$ and belongs to $\mathscr L(L^2(\sqrt \omega \dd\omega ))$.
    \item[(ii)] Consequently, it  generates a semigroup of contractions on $L^2(\sqrt \omega \dd\omega )$,  $S_t=e^{tL_\mu }$ such that
\begin{align*}
&\|L_\mu (S_tf)\|
 _{ L^2(\sqrt \omega \dd\omega ) }\le 
\frac{\|f\|
 _{ L^2(\sqrt \omega \dd\omega ) }}{e t},\,\,\forall f\in  L^2(\sqrt \omega \dd\omega ),\\
 &\lim_{t\to \infty}\|S_t(I-\Pi)f\|_{ L^2(\sqrt \omega \dd\omega ) }=0,\,\,\forall f\in  L^2(\sqrt \omega \dd\omega ),
  \end{align*}
  where $\Pi = \operatorname{Proj}_{\operatorname{Ker}(L_\mu)}$ is the orthogonal projection onto the $\operatorname{Ker}(L_\mu)$.
  \item[(iii)] For every $\mu \ge 0$, the operator $K_\mu $ defined in \eqref{eq: A_mu+K_mu} is not a compact operator from $L^2(\sqrt \omega \dd\omega )$ into itself. 
  \item[(iv)]\label{Theo:Theo2(iv)} For all $\mu >0$, the linear operator possesses a spectral gap. In particular, the Dirichlet form 
$D_\mu (g,g) = -\langle L_\mu  g, g\rangle_{ L^2(\sqrt \omega \dd \omega) }$  satisfies:
    $$ \exists \lambda_\mu >0;\,\,\, D_{\mu}(g) = - \langle L_\mu g,g \rangle \geq \lambda_\mu \| (I - \Pi)g\|_{L^2(\sqrt{\omega}\dd \omega)}^2, \quad \text{ for all } g \in L^2(\sqrt{\omega}\dd\omega). $$
\end{itemize}
\end{theorem}
\begin{remark}[On the functional space] \label{rem:on_functional_space}
    The relevant functional space here is the  $L^2(\sqrt{\omega}\dd\omega)$. 
For a radial function, the change of variables
$\omega=|\textbf{k}|^2$
gives
\begin{equation}
    \|\widetilde g\|_{L^2(\mathbb{R}^3)}^2
    =
    4\pi\int_0^\infty
    |g(r)|^2r^2 \dd r=
    2\pi\int_0^\infty
    |g(\omega)|^2\sqrt{\omega}\dd\omega.
\end{equation}
Thus, up to the constant $2\pi$,
the space $L^2((0,\infty);\sqrt{\omega}\,\dd\omega)
$
corresponds to the radial subspace of $L^2(\mathbb{R}^3)$. 
\end{remark}

\begin{remark} \label{rem:comparison_with_JFA}
Except for the last item, Theorem \ref{MainTh2} holds in the whole range $\mu\geq0$. Compared to the results of our previous paper \cite{EscMenegaki26} where the Dirac measure at zero is a basin of attraction for the linear semigroup, there is no contradiction as 
the two convergences hold in different topologies. The weight
$\sqrt{\omega}$ vanishes at the origin, and therefore condensation at the origin might be invisible in the weighted norm of $L^2((0,\infty);\sqrt{\omega}\dd\omega)$.
%}
\end{remark} 

Then, our main result on the nonlinear problem is the following.
\begin{theorem}
1\label{MainTh}
Let $\mu>0$ and $\lambda_\mu$ the spectral gap found in Theorem \ref{MainTh2}, (iv). For every $0<\delta<\lambda_\mu$, there exists 
$\rho = \rho(\mu,\delta)>0$ such that if 
  $$ g_0 \in (\operatorname{Ker}(L_\mu))^\perp\  \text{ and } \  \|g_0\|_{L^2(\sqrt{\omega})} \leq \rho,$$ then the perturbation equation  admits a unique global solution $g \in C^1( [0,\infty);L^2(\sqrt{\omega}))$ with   $g(0)=g_0$. 
    Moreover, $$g_t \in (\operatorname{Ker}(L_\mu))^\perp \ \text{ for all } t\geq 0 $$ and
    \begin{equation}
        \|g_t\|_{L^2(\sqrt{\omega})} \leq e^{- \delta t} \|g_0\|_{L^2(\sqrt{\omega})} \quad \text{ for all } t\geq 0.
    \end{equation}
\end{theorem}

Theorem \ref{MainTh} may be rephrased as follows in terms of the function $f$.

\begin{corollary}
\label{MainThB} 
Let $\mu>0$ and $0<\delta<\lambda_\mu$. 
For all initial data $f_0$ such that
\begin{align*}
&f_0(\omega )=\frac{1}{\omega + \mu}
\left(1+g_0(\omega) \right),\,\,\qquad \text{ a.e. }\, \omega \ge 0, \text{ where } \\
&g_0\in(\operatorname{Ker}(L_\mu))^\perp \ \text{ and }\quad \|(\mu +\omega )f_0-1\| = \|g_0\|_{L^2(\sqrt \omega \dd \omega )}\le \rho 
\end{align*}
where $\rho$ is as in Theorem \ref{MainTh}, there exists a unique global perturbative solution 
$f \in C^1( [0,\infty);L^2(\sqrt{\omega}\dd \omega))$ such that $f(0)=f_0$. 

For every $t\geq 0$, $(\mu +\omega )f_t-1 = g_t \perp \operatorname{Ker}(L_\mu)$ and
    \begin{equation}
      \|(\mu +\omega )f_t -1\|_{L^2(\sqrt{\omega})}=  \|g_t\|_{L^2(\sqrt{\omega})} \leq e^{- \delta t} \|(\mu +\omega )f_0-1\|_{L^2(\sqrt{\omega})} \quad \text{ for all } t\geq 0.
    \end{equation}
Finally, the dynamics preserve positivity, in the sense that if $f_0\geq 0$ then $f_t\geq 0$ for all $t>0$. 
\end{corollary}

\subsection{Difficulties and strategy of the proof}

The Rayleigh-Jeans equilibria, even though fundamental thermodynamic equilibria of wave kinetic equations and analogous to the Maxwellians in
the parallel kinetic theory of particles, have not been studied
sufficiently in the wave setting up to date, and basic properties on the whole frequency space, such as
their stability have remained open and, in fact, doubtful. Thus the perturbative theory around RJ is considerably less developed than the corresponding theory around Maxwellians. Maxwellians
are the entropy maximizers of the classical Boltzmann equation for gases under fixed mass, momentum and energy, and a mature global perturbative theory
around them has been established for several decades (cf. \cite{cercignani1988}, \cite{Villani_review} and literature therein.) 
The particularly difficult feature of Rayleigh-Jeans equilibria compared with Maxwellian states is their heavy tails and lack of integrability. We explain some of these difficulties below.

The equation \eqref{eq:NL evolution1} formally preserves the total wave action and energy of the system of waves, namely
$$
N(f)=\int _0^\infty f(t, \omega )  \omega^{1/2} \dd\omega;\,\,\, 
E(f)=\int _0^\infty f(t, \omega ) \omega^{3/2} \dd\omega.
$$ 
Both quantities are infinite for the RJ spectra. Furthermore, it is easy to check that for $f=n_\mu$, the integral term of  equation \eqref{eq:NL evolution1} diverges if the function $q(f)$ is replaced by any of the four terms $f f_2 f_3, f f_1 f_3,  f_1 f_2 f ,  f_1 f_2 f_3 $ or even by $ f f_2 f_3 + f f_1 f_3 $ or $ f_1 f_2 f + f_1 f_2 f_3$. The identity
$\mathcal{C}(n_\mu)=0$ must therefore be understood via the exact detailed-balance cancellation and not as the difference of two individually finite quantities.

The ultraviolet (high-frequency) divergence of the RJ distributions also has
a natural interpretation from quantum kinetic theory. This arises on physicist's arguments from the very origin of the wave kinetic operator when seen as the formal classical approximation of the bosonic Nordheim collision operator, true only in the high-occupation regime: $n(k)\gg 1$. Moreover, the Bose-Einstein equilibrium reduces to the Rayleigh-Jeans law when when the occupation numbers are large. But $n(k)$ decreases at high frequencies $k$, and thus the classical approximation stops being valid in that regime.

Thus the status of the RJ spectra as solutions of \eqref{eq:NL evolution1} may seem delicate.  
It is thus natural to ask whether solutions of \eqref{eq:NL evolution1} still exist for initial data given by perturbations of a Rayleigh-Jeans spectrum, for which perturbations, and in which topology the resulting dynamics is stable.
\vspace{-0.10cm}

\subsubsection{Strategy of the proof}

To investigate this problem of stability of RJ, we work in the natural-for-dissipation functional space $L^2((0,\infty);\sqrt{\omega}\dd\omega)$ This is good for two main reasons: 1) the equilibrium $n_\mu$ belongs in this space for $\mu>0$ and 2)  
the symmetrized Dirichlet form
of the linearised collision operator is naturally defined in this space, under radially symmetry (cf Remark \ref{rem:on_functional_space}).

We first prove that the linearised collision operator around each non-singular Rayleigh-Jeans thermodynamic equilibrium has a spectral gap in the relevant
weighted $L^2$-space. This proof as explained also in subsection \ref{subsec:NoCompact}, does not follow by the classical Grad
compact perturbation argument for angular-cutoff Boltzmann
operators. The reason is that, when setting the equation on the whole frequency space, and splitting the collisional operator as  $L_\mu= -A_\mu + K_\mu$, the operator $K_\mu$ is not compact. We prove this, by a scaling dilation argument, in Proposition \ref{prop:noncompactness}.

We nevertheless prove a Poincaré inequality for the linear operator $L_\mu$ and thus the existence of a spectral gap by combining a local compactness for $K_\mu$, coercivity properties at low and at high frequencies and a contradiction argument. In particular,
\begin{itemize}
    \item We prove that the multiplication part $A_\mu$ is lower bounded away from zero (cf Lemma \ref{lemm:A_mu}). 
    \item At low frequencies, for $\mu>0$, we show that a normalised sequence with Dirichlet form going to $0$, cannot concentrate close to $0$. More generally that  (cf Lemma \ref{lemm:Tightness_zero})
    $$\|\chi_{(0,M_1)}g\|_{L^2(\sqrt{\omega}\dd\omega)}^2
\lesssim_\mu -\langle L_\mu g,g \rangle + M_1^{3/2}\|g\|_{L^2(\sqrt{\omega}\dd\omega)}^2,
 \quad M_1 \downarrow 0.
$$
    \item The local operator $\chi_{[M_1,M_2]}K_\mu\chi_{[M_1,M_2]} $ is compact for $0<M_1,M_2<\infty$ (cf Lemma \ref{lem:local_K_mu_compact}).
    \item Finally we deal with the delicate high frequencies by providing an estimate of the form (cf Prop. \ref{Prop:SG_at_infty}) $$ \|\chi_{(M_2, \infty)}g\|_{L^2(\sqrt{\omega})}^2 \lesssim -\langle L_\mu g,g \rangle + \log(M_2)^{-2} \|g\|_{L^2(\sqrt{\omega})}^2, \quad M_2 \uparrow \infty, $$
    showing that a normalised sequence with Dirichlet form going to $0$, cannot concentrate close at infinity. 
    \item We then close with a contradiction argument. 
\end{itemize}

For the nonlinear problem, we first prove, by dyadic decomposition of the frequency space, the bilinear and trilinear bounds for $\mu>0$ (cf Prop. \ref{prop: control NL}):
$$ \|\Gamma_\mu(g,h)\|_{L^2(\sqrt \omega )}
    \lesssim_\mu \|g\|_{L^2(\sqrt \omega)}\|h\|_{L^2(\sqrt \omega)}, \quad  \|Q_\mu(g,h,\ell)\|_{L^2(\sqrt \omega)}
    \lesssim_\mu
    \|g\|_{L^2(\sqrt \omega)}\|h\|_{L^2(\sqrt \omega)}\|\ell\|_{L^2(\sqrt \omega)}.
$$
These estimates yield a local well-posedness by a fixed point argument.  
To obtain the nonlinear global-in-time well-posedness, we exploit the existence of a spectral gap that yields  exponential relaxation for sufficiently small perturbations of the equilibria. 

\subsubsection{Relation to previous results}

Formally the stability of RJ equilibria for $\mu =0$ in the whole space is considered in \cite{ZakLvovFalkov} for a large set of kinetic wave equations, but their results do not apply to the 3D cubic Schr\"odinger equation.

For Schr\"odinger waves on a finite box $B \subset \R^3$ it is proved in \cite{Men23} that the nonsingular RJ states (with $\mu >0$) are nonlinearly asymptotically stable under perturbations in $L^2(B)$. 
This frequency cutoff makes the $K_\mu$ part of the linearized operator immediately compact, and the spectral gap follows from a standard compact perturbation argument. Here we remove this assumption, lose the compactness but nevertheless prove via a more delicate analysis that the linear operator possesses a spectral gap, as explained above.

Rigorously, the case $\mu=0$ of singular RJ, for the equation \eqref{eq:NL evolution1} in the whole frequency space is considered in our previous work \cite{EscMenegaki26}. 
There, for suitable nonnegative initial data which agree with the singular RJ profile near the origin and are truncated at high frequencies, we prove the formation of a Dirac mass at zero in finite time. We also solve the initial value problem for the equation linearised around the singular RJ in several weighted 
functional spaces. For a class of nonnegative initial data, the solution converges weakly, as time goes to infinity, to a multiple of the Dirac mass at the origin. 
This can coincide with the findings of this paper, there  is no contradiction as the two behaviours concern different topologies. See also the discussion in Remark \ref{rem:comparison_with_JFA}.  

Insightful results on the local in time well-posedness of generalisations of equation \eqref{eq:NL evolution1} are also available.  
It was established in \cite{GIT20} for general radial dispersion relations in nearly critical weighted spaces, while in the Schr\"{o}dinger case, considered here, the authors prove local well-posedness in $L^2_s$ spaces with $s>1/2$. More recently the local theory was extended in \cite{IoakeimLeger26} to almost critical weighted $L^p$ spaces, for all $2\leq p\leq \infty$.
The initial data required in these papers need to decay faster at infinity than RJ. In particular, the RJ profiles lie at the excluded critical point of
these theories. So, these results do not provide a local well-posedness theory for the relative
$L^2(\sqrt{\omega}\dd\omega)$-neighbourhoods of RJ as  considered here.
 Interestingly, in \cite{Ampatzoglou_ill_posed} the authors provide a sharp ill-posedness/well-posedness threshold in terms of the collisional kernel of the operator, in weighted $L^\infty$ spaces. (The cubic
Schr\"{o}dinger WKE corresponds to the well-posed case).

Let us also mention the recent numerical work on the two-dimensional nonlinear Schr\"odinger model \cite{Laurie26}, where the long-time evolution exhibits an approximately Rayleigh-Jeans region $\left\{{T(t)}/(\omega+\mu(t) )\right\}_{\{t>0\}}$ over an increasingly broad range of frequencies, and a self similar (high-frequency) tail.
Also, formally the stability of RJ equilibria for $\mu =0$ in the whole space has been considered in \cite{ZakLvovFalkov} for a large set of kinetic wave equations, but their results do not apply to the 3D cubic Schr\"odinger equation.

Global weak solutions were constructed for the wave kinetic equation for a large class of initial measures with total finite mass, and their asymptotic behavior studied in \cite{EV13}.

For the spatially inhomogeneous wave kinetic equation, global well-posedness is known and long-time asymptotics have been obtained exploiting the dispersion introduced by the transport term. First,  in \cite{Ioakeim22} global existence and stability of mild solutions was proven near vacuum. Later this was extended to polynomially decaying initial data in \cite{AmpMillPavlovTaskov} and in \cite{AmpatzoglouLeger25} it was improved for translation invariant spaces in the spatial variable.

The stability of the nonequilibrium Kolomogorov-Zakharov steady states in $L^\infty$ for the inverse mass cascade has also been proved in \cite{CDG22} in the stationary setting. Regarding such non-equilibirum solutions, in \cite{SofferTran2020} time-dependent energy cascade solutions have been constructed rigorously for related three wave kinetic equations.

Regarding the analogous mathematical kinetic theory of phonons, the global stability of the nonsingular RJ spectra for the homogeneous kinetic FPUT equation, where the frequencies rather live on the torus and where there is no linear spectral gap, has been studied in \cite{GerLaMen26, GerLaMen26review}. For the inhomogeneous phonon Boltzmann, existence of solutions near vacuum for long times has been proved in \cite{Xiang}.

Finally, very recently the existence of local strong solutions in $L^1(\RR^d)$, $\dd\ge 2$ to the gravity water kinetic equation for initial data in a suitably weighted $L^2\cap L^\infty $ space has been proved in \cite{pan2026}.

 \subsection{Organisation of the article}
 Some first properties of the linear operator $L_\mu $ (boundedness, lack of compactness, collisional invariants and a time decay of the corresponding semigroup) are proved in Section \ref{sec:properties of L}. 
 
 In Section \ref{sec:LWP}, we control in $L^2(\sqrt \omega \dd\omega )$ the cubic and quadratic nonlinearities and  prove the local existence and uniqueness of solutions of the nonlinear equation \eqref{eq:NL evolution1} for initial data given by small perturbations 
 in $L^2(\sqrt \omega \dd\omega )$ of nonsingular RJ spectra. 
 
 In Section \ref{sec:SG} we prove the existence of a spectral gap for the linear operator $L_\mu$. 
 
 Finally in Section \ref{sec:GWP}, we deduce exponential decay in time of the linear semigroup and consequently the global existence of the solutions of equation \eqref{eq:NL evolution1} obtained in Section \ref{sec:LWP}.

 \subsection{Notation}
 We write $A \lesssim B$ if there exists a numerical constant $C$ so that $A\leq CB$, as well as $A\lesssim_\mu B$ if this constant depends on $\mu$. Also $A\sim B$ if $A\lesssim B$ and $B\lesssim A$.
 
 For a Banach space $X$, we denote by $\mathscr L(X)$, the space of bounded linear operators from $X$ into itself, endowed with the operator norm.

 Whenever we write $\langle \cdot, \cdot \rangle$, we mean the inner product with respect to the $L^2(\sqrt{\omega} \dd \omega)$ space, unless otherwise specified. 
 
Throughout the manuscript, we use tildes to denote functions and operators written in the frequency-vector variable $\textbf{k} \in \mathbb{R}^3$. Thus, $\widetilde g$ and $\widetilde n_\mu$ denote functions of $\textbf{k}$, while $g$ and $n_\mu$ denote their counterparts in the radial energy variable $\omega=|\textbf{k}|^2$. 

We use the same convention for operators: for instance, $\widetilde L_\mu$ and $\widetilde Q_\mu$ act on functions of $\textbf{k}$, whereas $L_\mu$ and $Q_\mu$ act on functions of $\omega$.

Also, we denote by $\Pi_\mu$ (or simply by $\Pi$) the $L^2(\sqrt{\omega}\dd \omega)$-orthogonal projection onto $\operatorname{Ker}(L_\mu)$.

\subsection{Acknowledgments} Angeliki Menegaki acknowledges support from the Engineering and Physical Sciences Research Council fellowship with reference UKRI2025. Miguel Escobedo is supported by MINECO through grant PID2023-146872OB-I00.
We thank Pierre Germain and Clément Mouhot for useful discussions.  

\noindent
\textbf{AI use statement.}
AI tools were used for proofreading and for supplying the proof of Proposition 4.3, in particular for the idea of transforming the problem to the exponential variables.

\section{Properties of the linear operator $L_\mu$} \label{sec:properties of L}

\subsection{Boundedness of $L_\mu$ and lack of compactness in $L^2(\sqrt \omega d\omega )$ } \label{subsec:Boundedness_NoComp_KerL}

We work in the space $L^2((0, \infty), \sqrt{\omega} \dd \omega)$ as this is the natural space when it comes to energy dissipation. We easily see that the operator is self-adjoint in that space and, using the symmetry under permutations of the four collision variables, we see that $L_\mu$ also has a sign. Indeed for $g \in C_c^\infty(0, \infty)$: 
\begin{equation}
\begin{split}
   \mathcal{D}_{\mu}(g,g) &:= - \langle L_\mu g, g \rangle_{L^2(\sqrt{\omega} \dd \omega)} \\
  - \langle L_\mu g, g \rangle_{L^2(\sqrt{\omega} d \omega)} & = \frac{1}{4} \iiint_{\omega_1+\omega_2\geq \omega} \min (\sqrt{\omega_1}, \sqrt{\omega_2}, \sqrt{\omega_3}, \sqrt{\omega}) \prod_{\ell=0}^3 n(\omega_\ell) \times 
    \\
    & \times \left[- \frac{g(\omega_1)}{n (\omega_1)} - \frac{g(\omega_2)}{n(\omega_2)} + \frac{g(\omega_3)}{n (\omega_3)} + \frac{g(\omega)}{n(\omega_0)} \right]^2 \dd \omega \dd \omega_1 \dd \omega_2 \geq 0, 
\end{split}
    \end{equation}
where the notation is that $\omega_0=\omega$. 

\subsubsection{Boundedness of $L_\mu$} \label{subsec:BoundedL_mu}
In this subsection, we prove that the linear operator is bounded as an operator from $L^2(\sqrt{\omega} \dd \omega)$ to itself, for any $\mu\geq 0$: that is regardless of whether we linearise around singular or not Rayleigh-Jeans equilibria.   

\begin{proposition}
    Let $\mu \geq 0$. The operator $L_\mu$ is bounded and symmetric in $L^2((0, \infty); \sqrt{\omega} \dd \omega)$. In particular, there exists $C_\mu>0$ so that  
    $$
\|L_\mu g\|_{L^2((0, \infty); \sqrt{\omega} \dd \omega) }\leq C_\mu\|g\|_{L^2((0, \infty); \sqrt{\omega} \dd \omega)},
\quad  \text{for all } g\in L^2((0, \infty), \sqrt{\omega} \dd \omega).$$
\end{proposition}

\begin{proof}
First we assume that $g \in C_c^\infty(0,\infty)$, so that the symmetrisation procedure above is well-justified (Fubini is allowed). We will in the end extend the operator $L_\mu$ uniquely to a bounded self-adjoint operator on $L^2(\sqrt{\omega})$. 

    As a first step, we bound the square in the integrand as follows: 
  \begin{align*}
  \Bigg[- \frac{g(\omega_1)}{n(\omega_1)} - \frac{g(\omega_2)}{n(\omega_2)} + &\frac{g(\omega_3)}{n(\omega_3)} + \frac{g(\omega_0)}{n(\omega_0)} \Bigg]^2  \leq \\
 &\leq   4 \left( 
    \left\vert \frac{g(\omega_0)}{n(\omega_0)}\right\vert^2 + \left\vert \frac{g(\omega_1)}{n(\omega_1)}\right\vert^2 + \left\vert \frac{g(\omega_2)}{n(\omega_2)}\right\vert^2+ \left\vert \frac{g(\omega_3)}{n(\omega_3)}\right\vert^2 \right)  
    \end{align*}
     which implies that 
  \begin{align*}
  &  \langle L_\mu g, g \rangle_{L^2(\sqrt{\omega} \dd \omega)} \leq \\
    &\leq \sum_{\ell=0}^3 
    \iiint_{\omega_1+\omega_2\geq \omega} \min (\sqrt{\omega_1}, \sqrt{\omega_2}, \sqrt{\omega_3}, \sqrt{\omega_0}) \prod_{j=0}^3 n(\omega_j) \left\vert \frac{g(\omega_\ell)}{n(\omega_\ell)}\right\vert^2 \dd \omega_0 \dd \omega_1 \dd \omega_2 
    \end{align*}
    and using the symmetry in all four variables of the integral kernel, it suffices to bound the first term 
    \begin{equation}
\begin{split}
I_0 &:=  \iiint_{\omega_1+\omega_2\geq \omega} \min (\sqrt{\omega_1}, \sqrt{\omega_2}, \sqrt{\omega_3}, \sqrt{\omega_0}) \prod_{\ell=0}^3 n(\omega_\ell) \left\vert \frac{g(\omega_0)}{n(\omega_0)}\right\vert^2 \dd \omega_0 \dd \omega_1 \dd \omega_2 \\
&  = 
\int_0^\infty \frac{\omega_0 +\mu}{\sqrt{\omega_0}} |g(\omega_0)|^2\times \\
&\hskip 0.6cm \times  \left(\iint_{\omega_1+\omega_2 \geq \omega_0} \min (\sqrt{\omega_1}, \sqrt{\omega_2}, \sqrt{\omega_3}, \sqrt{\omega_0}) \prod_{\ell=1}^3 n(\omega_\ell) \dd \omega_1 \dd \omega_2\right) \sqrt{\omega_0}\dd \omega_0. 
    \end{split}
    \end{equation}
    We are then going to show that in fact the interior double integral is uniformly bounded in $\omega_0>0$, i.e. that 
    \begin{align} \label{eq: unif bound}
    \sup_{\omega_0>0}\ \frac{\omega_0 +\mu}{\sqrt{\omega_0}} 
 \iint_{\omega_1+\omega_2 \geq \omega_0} \min (\sqrt{\omega_1}, \sqrt{\omega_2}, \sqrt{\omega_3}, \sqrt{\omega_0}) \prod_{\ell=1}^3 n(\omega_\ell) \dd \omega_1 \dd \omega_2 \leq C_\mu <\infty
    \end{align}
    for some finite constant $C_\mu$. Then, given \eqref{eq: unif bound},  we are ready to conclude since for $g \in C_c^\infty(0,\infty)$:  
    $$I_0 \leq C_\mu \|g\|_{L^2((0, \infty), \sqrt{\omega} \dd \omega)}^2, $$
    which by symmetry implies that  
    $$ \langle L_\mu g, g \rangle_{L^2(\sqrt{\omega} \dd \omega)} \leq \tilde{C}_\mu \|g\|_{L^2((0, \infty), \sqrt{\omega} \dd \omega)}^2.$$
    Then polarisation gives  that for $g, h \in C_c^\infty (0,\infty)$ we have 
    $$ \vert \langle L_\mu g, h \rangle_{L^2(\sqrt{\omega} \dd \omega)} \vert
    \leq \tilde{C}_\mu \|g\|_{L^2(\sqrt{\omega} \dd \omega)} \|h\|_{L^2(\sqrt{\omega} \dd \omega)}. $$
    Since $C_c^\infty(0,\infty)$ is dense in $L^2(\sqrt{\omega} d \omega)$, the bilinear form $\mathcal{D}_\mu(g,h)$ extends uniquely by density to a bounded bilinear form on all of $L^2(\sqrt{\omega})\times L^2(\sqrt{\omega})$. So by Riesz representation, there exists a unique bounded operator still denoted by $L_\mu$ so that 
    $$ \mathcal{D}_\mu(g,h) = \langle L_\mu g, h \rangle_{L^2(\sqrt{\omega} \dd \omega)} \quad \text{ for all } g,h \in L^2(\sqrt{\omega} \dd \omega).$$
    
    Also there holds 
    $$ \|L_\mu g \|_{L^2(\sqrt{\omega}\dd \omega)} \leq \tilde{C}_\mu \|g\|_{L^2((0, \infty), \sqrt{\omega} \dd \omega)}.$$
    
    In the rest of the proof we show the \eqref{eq: unif bound}:

    \noindent
    \underline{\textbf{Case 1: Let $\mu>0$.}} We notice first that without loss of generality we may assume that $\mu=1$. Indeed otherwise we can rescale $\omega = \mu x, \omega_1 = \mu y, \omega_2= \mu z $ which gives that 
    \begin{align*}
        &\frac{\omega_0 +\mu}{\sqrt{\omega_0}} 
 \iint_{\omega_1+\omega_2 \geq \omega_0} \frac{\min (\sqrt{\omega_1}, \sqrt{\omega_2}, \sqrt{\omega_3}, \sqrt{\omega_0})}{(\omega_1+\mu)(\omega_2+\mu)(\omega_3+\mu)} \dd  \omega_1 \dd \omega_2  =\\
 & \hspace{5cm}
 \frac{x+1}{\sqrt{x}} 
 \iint_{y+z \geq x} \frac{\min (\sqrt{y}, \sqrt{z}, \sqrt{y+z-x}, \sqrt{x})}{(y+1)(z+1)(y+z-x+1)} \dd y \dd z. 
    \end{align*}  
    \underline{Now if $x\geq 1$:}  We set $$
y=xa,\qquad z=xb,\qquad c:=a+b-1.$$
Then $a,b>0$, $c>0$,  and
$ y+z-x=xc$, $dydz=x^2\ da db.$ Moreover,
$$
\min(\sqrt{x},\sqrt y,\sqrt z,\sqrt{y+z-x})
=
\sqrt{x}\ \min(1,\sqrt a,\sqrt b,\sqrt c),
$$
and the denominator becomes  
$$(1+y)(1+z)(1+y+z-x)
=
x^3(a+x^{-1})(b+x^{-1})(c+x^{-1}).$$
Thus the quantity we want to estimate is written as 
$$(1+x^{-1})
\iint_{\substack{a,b>0\\ a+b>1}}
\frac{
\min(1,\sqrt a,\sqrt b,\sqrt{a+b-1})
}
{(a+x^{-1})(b+x^{-1})(a+b-1+x^{-1})}
\,\dd a\,\dd b.
$$
Since now $x> 1$, we have $1+x^{-1}\le 2$, and the denominator is  lower bounded by  
$ab(a+b-1).$ Altogether give the upper bound 
$$
2
\iint_{\substack{a,b>0\\ a+b>1}}
\frac{
\min(1,\sqrt a,\sqrt b,\sqrt{a+b-1})
}
{ab(a+b-1)}
\,\dd a\,\dd b.
$$
This is now finite. Indeed, 
the only possible singularities occur at $a=0$, $b=0$, or $c=0$. When  $a=0$ and  $c=0$, we have  $b=1+c-a\sim 1$, and the integrand is upper-bounded by
$$
C\ \frac{\min(\sqrt a,\sqrt c)}{ac}.
$$
This is, locally around $0$, integrable, since splitting in cases  $a<c$ and $a>c$ we have: 
$$
\int_0^\delta\int_0^\delta
\frac{\min(\sqrt a,\sqrt c)}{ac}\,\dd a\,\dd c
=
\int_0^\delta\int_0^c
\frac{1}{\sqrt a\,c}\,\dd a\,\dd c
+
\int_0^\delta\int_0^a
\frac{1}{a\sqrt c}\,\dd c\,\dd a
$$
where we changed the integration variables from $(a,b)$ to $(a,c)$.
The right-hand side is then finite, for example
$$
\int_0^\delta\int_0^c
\frac{1}{\sqrt a\,c}\,\dd a\,\dd c
=
2\int_0^\delta c^{-1/2}\,\dd c <\infty,
$$
and the second term is identical. The case when $b=0$ and $c=0$ is analogous by symmetry.

Away from the $c=0$ - case: we cannot have both $a$ and $b$ to be close to $0$ due to the restrictions of the domain (we need $a+b>1$). \\
When only $a$ is close to $0$, the singularity is of order $a^{-1/2}$ which is integrable as well.\\
By symmetry the same happens when only $b$ is close to $0$.\\
When $a\sim 0$ while $b \to \infty$ we have $c \to \infty$ and the integrand is
$$ \frac{
\min(1,\sqrt a,\sqrt b,\sqrt c)
}
{abc} \sim \frac{\min(\sqrt{a})}{abc} \sim \frac{1}{\sqrt{a}b^2}$$
for which the integral is also finite: $$ \int_0^1 \frac{da}{\sqrt{a}}  \int_1^\infty \frac{db}{b^2} <\infty. $$
\\
Finally for large $a,b$, we have 
$$
\frac{
\min(1,\sqrt a,\sqrt b,\sqrt c)
}
{abc}
\le
\frac{1}{ab(a+b-1)}
\lesssim
\frac{1}{ab(a+b)}, 
$$
which is integrable when $a,b\ge 1$,  since $$ \iint_{a,b\ge 1}\frac{1}{ab(a+b)}\,\dd a\,\dd b
\le
2\int_1^\infty\int_1^a \frac{1}{a^2b}\,\dd b\,\dd a =
2\int_1^\infty \frac{\log a}{a^2}\,\dd a $$ 
which is finite. We get then that 
$$
\sup_{x >1} \frac{x+1}{\sqrt{x}} 
 \iint_{y+z > x} \frac{\min (\sqrt{y}, \sqrt{z}, \sqrt{y+z-x}, \sqrt{x})}{(y+1)(z+1)(y+z-x+1)} \dd y \dd z < \infty.
 $$
 
\underline{Now if $0<x<1$:} Since the minimum is always less than $\sqrt{x}$ we may bound our main quantity by 
$$ 
(x+1) 
 \iint_{y+z > x} \frac{\dd y  \dd z }{(y+1)(z+1)(y+z-x+1)}
 \lesssim \iint_{y+z > x} \frac{\dd y \dd z }{(y+1)(z+1)(y+z+1)}
$$
where in the inequality we applied the bound $y+z+1\leq 2 (y+z-x+1)$, which holds since $y+z>x$ and $x<1$. Then we write for $s=y+z$:
$$ \int_{0}^{\infty} \frac{\dd s}{(s+1)} \int_{y=0}^s \frac{\dd y }{(y+1)(s-y+1)}  = 2 \int_{0}^{\infty} \frac{\log(s+1)}{(s+2)(s+1)} \dd s  = \frac{\pi^2}{6}<\infty
$$ since $s>0$, and then we conclude because the last integral is finite. 

Altogether we have that indeed $$\sup_{x >0}\ \frac{x+1}{\sqrt{x}} 
 \iint_{y+z > x} \frac{\min (\sqrt{y}, \sqrt{z}, \sqrt{y+z-x}, \sqrt{x})}{(y+1)(z+1)(y+z-x+1)} \dd y \dd z \quad \text{ is finite}. $$
 
 \noindent
    \underline{\textbf{Case 2: Let $\mu=0$.}} As before we rescale: $\omega_1= a \omega$, $\omega_2= b \omega$, $\omega_3= \omega(a+b-1)$, for $a,b>0$ and $a+b>1$. Then the quantity we aim to estimate uniformly in $\omega_0$, is 
    \begin{align*}
        \sqrt{\omega_0} 
 \iint_{\omega_1+\omega_2 > \omega_0} 
 \frac{ \min ( \sqrt{\omega_1}, \sqrt{\omega_2}, \sqrt{\omega_3}, \sqrt{\omega_0})}{\omega_1 \omega_2 \omega_3} \dd \omega_1 \dd \omega_2= \\
 = 
 \iint_{\substack{a,b>0\\ a+b>1}}
\frac{
\min(1,\sqrt a,\sqrt b,\sqrt{a+b-1})
}
{ab(a+b-1)}
\,\dd a\,\dd b, 
    \end{align*}
    since we check that all the powers of $\omega_0$ cancel. This integral was already examined above, as it is the case where the large $x\geq 1$ case essentially boiled down to. 
    Repeating verbatim, we conclude that it is finite. 
    %Now we examine this integral and we see that the only possible singularities come when $a=0, b=0, a+b-1=0$. Also, since $a+b>1$, we cannot have the same time both $a$ and $b$ to be $0$. \\
    %When $a\sim 0$ and $a+b\sim 1$, then $b \sim 1$ and the integrand is bounded by 
    %$$ C \frac{\min (\sqrt a,\sqrt{a+b-1}) }{ab (a+b-1)} da db.$$
    %Exactly as was done above, this integral is finite. \\
    %When $b=0$ and $a+b\sim 1$ the same.
    We conclude then that \eqref{eq: unif bound} holds for all $\mu \geq 0.$   \end{proof}

\subsubsection{Non Compactness of $K_\mu$} \label{subsec:NoCompact}

In this subsection we prove that approaching the problem of relaxation towards equilibrium in a perturbative way around the multiplication operator, is not (!) the proper machinery to tackle the problem. 
In particular we show that after splitting $L_\mu$ into a multiplication $A_\mu$ and a remaining part $K_\mu$ as written in \eqref{eq: A_mu+K_mu}, is not helpful since the operator $K_\mu$ is not compact in $L^2(\sqrt{\omega} \dd \omega)$.

This comes in contrast to the existing literature in classical kinetic theory of particles but also phonons, where understanding the spectrum of the linearised operator boils down to understanding the multiplicative part, since the remaining operator $K_\mu$ associated to these systems, can be proven to be in fact compact and thus leaves the spectrum unaffected (modulo creating some additional discrete spectrum by Weyl), \cite{GradUnesco, LaureSR09, Villani_review, cercignani1988, Clement06, Clement07} for classical references in the study of the linear Boltzmann operator in the kinetic theory particles,  \cite{LukkarinenSpohn2008, GerLaMen26, GerLaMen26review}, in the study of kinetic theory of phonons, but also \cite{Men23} for a truncated version of the same kinetic wave operator considered here, where compactness of the remaining part $K_\mu$ holds. Comparing our result here with \cite{Men23}, shows that the lack of compactness comes due to high frequencies which exactly is reflected in the proof below.  

To start with, we remind that, for $\mu>0$ the operator $K_\mu: L^2(\sqrt{\omega}) \to L^2(\sqrt{\omega})$ is the operator 
\begin{align*}
    (K_{\mu} g )(\omega) &= \iint_{D(\omega)} \frac{\operatorname{min}(\sqrt{\omega},\sqrt{\omega_1}, \sqrt{\omega_2}, \sqrt{\omega_3})}{\sqrt{\omega}\prod_{\ell=1}^3 (\omega_\ell+\mu)} \Bigg(2 g(\omega_1) (\omega_1 + \mu)  - g(\omega_3)(\omega_3 + \mu)  \Bigg) \dd \omega_1 \dd \omega_2
\\
& =: (K_{1,\mu}g)(\omega) - (K_{2,\mu}g)(\omega)
\end{align*}
where $D(\omega) = \{(\omega_1, \omega_2) \in (0,\infty)^2: \omega_1+ \omega_2 \geq  \omega\}$. 
The main result of this subsection is: 

\begin{proposition} \label{prop:noncompactness}
  For all $\mu  \geq 0$, the operator $K_\mu$ is not compact on $L^2(\sqrt{\omega} \dd \omega)$.
\end{proposition}

For the proof we are going to use the following lemma.

\begin{lemma}
    Let $0\leq f \in C_c^\infty((0,\infty))$ with $\operatorname{supp}f \subset[a,b]$, where $0<a<b<\infty$. Then $$ K_\mu f \to K_0 f \ \text{ a.e. as } \mu \downarrow 0^+. $$
\end{lemma}

\begin{proof}
 We have 
$$[K_\mu f](\omega) = \iint_{D(\omega)} \Omega (\omega, \omega_1, \omega_2, \omega _3) n_{\mu,1} n_{\mu,2} n_{\mu,3} \Big[- \frac{f_3}{n_{\mu,3}} + \frac{f_1}{n_{\mu,1}} + \frac{f_2}{n_2} 
\Big] \dd \omega_1 \dd \omega_2 =: -K_{\mu,3}f+K_{\mu,1}f+K_{\mu,2}f.
$$ 
Since $f \geq 0$ by assumption, all three integrands are nonnegative and also $n_{\mu,j} = \frac{1}{\omega_j+\mu}$ is increasing in $\mu$ and $ n_{\mu,j}  \uparrow n_{0,j}=\frac{1}{\omega_j}$ as $\mu\downarrow 0$. And so, each of the integrands increases to the corresponding integrands for $\mu=0$. The Monotone
Convergence Theorem then gives
$$  K_{\mu,j}f(\omega)  \to K_{0,j}f(\omega), \quad \text{ for each } j=1,2,3.
$$
It remains to check that the three individual integrals $K_{0,j}f$ are finite so that it makes sense to add and subtract them in order to conclude that $ K_{\mu}f(\omega) \to K_{0}f(\omega)$. We start with $K_{0,1}f$: 
$$ 
K_{0,1}f = \iint_{D(\omega)} \Omega (\omega, \omega_1, \omega_2, \omega _3) \frac{f_1}{\omega_2 \omega_3} \dd \omega_1 \dd\omega_2 \leq \frac{\|f\|_{\infty}}{\sqrt{\omega}} \iint_{\substack{\omega_2, \omega_3 >0\\ a\leq \omega+\omega_3-\omega_2\leq b}} \frac{ \operatorname{min}(\sqrt{\omega_2}, \sqrt{\omega_3})}{ \omega_2 \omega_3} \dd \omega_2 \dd\omega_3
$$
where we changed the variables from $(\omega_1, \omega_2) \to (\omega_2, \omega_3)$. Now the last integral is finite, since near $0$: 
$$ \int_0^1\int_0^1
\frac{\min ( \sqrt{\omega_2},\sqrt{\omega_3})}
{\omega_2\omega_3}
\dd \omega_2 \dd \omega_3 = 2\int_0^1\int_0^{\omega_3}
\frac{1}{\sqrt{\omega_2} \omega_3}
\dd \omega_2 \dd \omega_3 < \infty.$$
Same argument holds for the term $K_{0,2}f$ by symmetry. Finally for $K_{0,3}f$: 
$$ K_{0,3}f = \iint_{D(\omega)} \Omega (\omega, \omega_1, \omega_2, \omega _3) \frac{f_3}{\omega_1 \omega_2} \dd \omega_1 \dd\omega_2 \leq \frac{\|f\|_{\infty}}{\sqrt{\omega}} \iint_{\substack{ a<\omega_3<b\\0<\omega_1 < \omega + \omega_3}} \frac{ \operatorname{min}(\sqrt{\omega_1}, \sqrt{\omega+\omega_3-\omega_1})}{ \omega_1 (\omega+\omega_3-\omega_1)} \dd \omega_1 \dd\omega_3
$$
where we changed the variables from $(\omega_1, \omega_2) \to (\omega_1, \omega_3)$. 
This is again finite, arguing as above for the possible singularities at zero, since 
$$  K_{0,3}f \lesssim \frac{\|f\|_{\infty}}{\sqrt{\omega}} \int_a^b \frac{\dd \omega_3}{\sqrt{\omega+\omega_3}}<\infty.$$
\end{proof}
\begin{proof}[Proof of Proposition \ref{prop:noncompactness}]
 %We will use the fact that $K_\mu f \to K_0 f$ a.e. as $\mu \downarrow 0^+$ for one $f \in C_c^\infty([a,b])$ for some $a>0, b <\infty$. This result we can prove earlier and cite it.
    
 Let $\mu>0$: We define the unitary operator 
$U_\lambda : L^2(\sqrt{\omega}\dd \omega) \to L^2(\sqrt{\omega} \dd \omega)$ by the dilation  $(U_{\lambda}f)(\omega):= \lambda^{-\frac{3}{4}} f \left( \frac{\omega}{\lambda}\right)$. This is indeed unitary in this space since: 
$$ \int_0^\infty |U_{\lambda}f|^2 (\omega) \sqrt{\omega}\dd \omega = \int_0^\infty |f|^2(x) \sqrt{x} \dd x $$
after changing the variable $\omega = \lambda x$. 
We now notice that the operator $K_{\tfrac{\mu}{\lambda}}$ is the operator defined by the conjugation $$U_\lambda^{-1} K_\mu U_\lambda = K_{\tfrac{\mu}{\lambda}}.$$ Indeed: We look for example at the first piece $K_{1,\mu}$ (the second piece is similar) and we write 
\begin{align*} 
[K_{1,\mu} U_\lambda g](\omega) &= 2 \iint_{D(\omega)} \frac{\operatorname{min}(\sqrt{\omega},\sqrt{\omega_1}, \sqrt{\omega_2}, \sqrt{\omega_3})}{\sqrt{\omega}\prod_{\ell=2}^3 (\omega_\ell+\mu)} \lambda^{-\frac{3}{4}}  g\left(\frac{\omega_1}{\lambda}\right)\  \dd \omega_1 \dd \omega_2 \\ & = 
 2 \lambda^{-\frac{3}{4}}\iint_{D(z)} 
  \frac{\operatorname{min}(\sqrt{z},\sqrt{x}, \sqrt{y}, \sqrt{x+y-z})}{\sqrt{z} (y+\tfrac{\mu}{\lambda})(x+y-z+\tfrac{\mu}{\lambda})} g(x)
  \ \dd x \dd y 
    \\ & = \lambda^{-\frac{3}{4}} \big[ K_{1,\tfrac{\mu}{\lambda}}  g \big] \left( \frac{\omega}{\lambda}\right) = \big[U_\lambda K_{1,\tfrac{\mu}{\lambda}}\ g \big](\omega) 
    \end{align*}
    for $\omega_1 = \lambda x, \omega_2 = \lambda y, \omega = \lambda z$, since the Jacobian cancels the $\lambda^2$ factor from the denominator. 
    Now we proceed by a contradiction argument. We consider a smooth function localised away from the origin, say $0\leq f \in C_c^\infty([a,b])$, with $K_0 f \neq 0$. Then we define the sequence 
    \begin{align}
    \psi_n:= U_{\lambda_n} f \quad \text{ with the parameter } \ \lambda_n \to \infty, \text{ as } \ n \to \infty.\label{EPsin}
    \end{align}
     Since the operator is unitary we have for all $n$, 
 $$
\|\psi_n\|_{L^2(\sqrt{\omega}d\omega)} = \| U_{\lambda_n}f\|_{L^2(\sqrt{\omega}d\omega)} = \|f\|_{L^2(\sqrt{\omega}d\omega)},
 $$
    and then 
   \begin{align}
   \psi_n \rightharpoonup 0\,\,\text{in}\,\, L^2(\sqrt{\omega}d\omega)\,\,\text{as}\,\, n \to \infty
   \end{align}
  since $f$ is supported away from zero. 
    Now, if the operator $K_\mu$ was compact, we would have that $K_\mu \psi_n \to 0$ strongly, or equivalently using the properties discussed above, compactness would imply that 
    \begin{equation} \label{eq:compact_contradict}
    \begin{split}
    \| K_{\tfrac{\mu}{\lambda_n}} f \|_{L^2(\sqrt{\omega}d\omega)} &= \| U_{\lambda_n} K_{\tfrac{\mu}{\lambda_n}} f \|_{L^2(\sqrt{\omega}d\omega)} 
    \\ &= \| K_\mu U_{\lambda_n} f \|_{L^2(\sqrt{\omega}d\omega)} = \|K_\mu \psi_n \|_{L^2(\sqrt{\omega}d\omega)} \to 0.
    \end{split}
    \end{equation} 
    But we also have that for a.e. $\omega \in (0,\infty)$:  $\Big[ K_{\tfrac{\mu}{\lambda_n}} f\Big] (\omega) \to [K_0 f] (\omega)$ for the $f$ we considered here.
    
    Therefore
    \begin{align*}
        \int_0^\infty |K_0 f|^2 \sqrt{\omega} d \omega = \int_0^\infty \liminf_{n \to \infty} \left\vert K_{\tfrac{\mu}{\lambda_n}} f\right\vert^2 \sqrt{\omega} \ d \omega \leq  \liminf_{n \to \infty} \int_0^\infty  \left\vert K_{\tfrac{\mu}{\lambda_n}} f\right\vert^2 \sqrt{\omega} \ d \omega =0 
    \end{align*}
    where for the inequality we applied Fatou's lemma and the last limit follows from \eqref{eq:compact_contradict}. This is a contradiction since the left-hand side is  strictly positive. Thus the operator $K_{\mu}$ cannot be compact in $L^2(\sqrt{\omega}d\omega)$. 
    
    If we suppose  that $\mu =0$ then 
    $$U_\lambda^{-1} K_0 U_\lambda = K_{0}.$$
 As in the previous case, if the operator $K_0$ was compact, we would have $\|K_0\psi_n\|_{ L^2(\sqrt \omega d\omega ) }\to 0$ as $n\to \infty$ and 
     \begin{equation} \label{eq:compact_contradict2}
    \begin{split}
    \| K_0 f \|_{L^2(\sqrt{\omega}d\omega)} &= \| U_{\lambda_n} K_0 f \|_{L^2(\sqrt{\omega}d\omega)} 
    \\ &= \| K_0 U_{\lambda_n} f \|_{L^2(\sqrt{\omega}d\omega)} = \|K_0 \psi_n \|_{L^2(\sqrt{\omega}d\omega)} \to 0.
    \end{split}
    \end{equation} 
That contradiction shows that $K_0$ is not compact. 
 \end{proof}

\subsection{Collisional Invariants and some first time decay properties for the linear semigroup}

We start with the following proposition.

\begin{proposition}
The operator $-L_\mu$ is m-accretive on $L^2(\sqrt{\omega}\dd \omega)$.  
\end{proposition}

\begin{proof}
For all $\lambda >0$,
\begin{align*}
\| f-\lambda L_\mu (f) \|_{L^2(\sqrt{\omega})}^2=    
\| f \|_{L^2(\sqrt{\omega})}^2+\lambda^2  \| L_\mu(f) \|_{L^2(\sqrt{\omega})}^2-2 \lambda\langle f,L_\mu f\rangle_{L^2(\sqrt{\omega})} \ge \|f\|_{L^2(\sqrt{\omega})}^2
\end{align*}
since $\langle f, L_\mu f\rangle_{L^2(\sqrt{\omega})}\leq 0$.
Then $-L_\mu$ is accretive. Moreover, for any $f\in L^2(\sqrt \omega d\omega )$ and $\lambda>0$,there exists $g\in L^2$ such that $g-\lambda L_\mu(g)=f$. If we define $a(g, h)=\langle g-\lambda L_\mu(g), h\rangle $ this is a bilinear form on 
$ L^2(\sqrt \omega \dd\omega )\times L^2(\sqrt \omega \dd\omega )$ continuous and coercive since 
$a(g, g)=\|g\|_{L^2(\sqrt{\omega})}^2-\lambda \langle L_\mu(g), g)\rangle_{\sqrt{\omega}} \ge \|g\|_{L^2(\sqrt{\omega})}^2$. By Lax-Milgram there exists $g\in L^2(\sqrt{\omega} \dd \omega)$ such that $a(g, h)=\langle f, h\rangle$ for all $h\in L^2(\sqrt{\omega} \dd \omega)$ and then such that $g-\lambda L_\mu(g)=f$.
\end{proof}
Since the operator $L_\mu$ is symmetric on $L^2(\sqrt{\omega} \dd \omega)$, by standard semigroup theory we get the following corollary. 

\begin{corollary} \label{cor: decay_contr}
For every $\mu \ge 0$, the operator $L_\mu$ generates a strongly continuous semigroup of contractions $e^{t L_\mu}$ in $L^2(\sqrt{\omega} \dd \omega)$. For all $f\in L^2(\sqrt{\omega} \dd \omega)$, the function $u(t)=e^{t L_\mu}f$ satisfies:
\begin{itemize}
    \item[(i)] The Cauchy problem 
    \begin{align*}
\begin{cases}
&\frac{\partial u}{\partial t}=L_{\mu} u,\,\,\forall t>0\\
&u(0)=f, 
\end{cases}
\end{align*}
possesses a unique solution $u\in C([0, \infty),L^2(\sqrt{\omega} \dd \omega))\cap C^1((0,\infty),L^2(\sqrt{\omega} \dd \omega)) $.  
\item[(ii)] It holds that
\begin{align*}
&\| L_\mu(u(t))\|_{L^2(\sqrt{\omega} \dd \omega)} \le \frac{1}{\sqrt 2\, t} \|f\|_{L^2(\sqrt{\omega} \dd \omega)},\,\forall\ t>0, \text{ and } \\
&\int_0^{\infty} s \| L_ \mu u(s)\|^2 _{ L^2(\sqrt \omega d\omega ) }\dd s\le
 \frac{1}{4}\| f \|^2_{ L^2(\sqrt \omega \dd\omega ) }.
 \end{align*}
\item[(iii)] For all $f\in L^2(\sqrt{\omega} \dd \omega)$
 \begin{align*}
\lim _{ t\to \infty } \|S_t(I-\Pi)f \|_{L^2(\sqrt{\omega} \dd\omega )}=0.
\end{align*}
\end{itemize}
\end{corollary}

\begin{proof}
The existence of the semigroup of contractions and properties (i), (ii) follow from classical results (cf. \cite {brezis2011}).

Now for (iii): For every $f\in L^2(\sqrt{\omega}\dd\omega)$,  
$(I-\Pi)f \in \big( \operatorname{Ker}(L_\mu ) \big)^\perp$ and as $L_\mu$ is self-adjoint in $L^2(\sqrt{\omega})$, $\overline{\operatorname{Ran}(L_\mu )} =\big( \operatorname{Ker}(L_\mu) \big)^\perp$. By basic properties of closure, for all $\varepsilon>0$ we may find $h \in L^2(\sqrt{\omega}\dd \omega)$ so that $$ \| (I-\Pi)f - L_\mu h\|_{L^2(\sqrt{\omega}\dd \omega)} < \varepsilon. $$
We turn now to the quantity we want to estimate:
\begin{equation}
    \begin{split}
        \|S_t (I- \Pi)f\|_{L^2(\sqrt{\omega}d \omega)} & \leq \|S_t ( (I- \Pi)f - L_\mu h) \|_{L^2(\sqrt{\omega}\dd \omega)} + \| L_\mu S_t  h\|_{L^2(\sqrt{\omega}\dd \omega)}\\
        & \leq \| (I- \Pi)f - L_\mu h \|_{L^2(\sqrt{\omega}\dd \omega)} + \| L_\mu S_t  h\|_{L^2(\sqrt{\omega}\dd \omega)}
        \\ & < 
        \varepsilon + \| L_\mu e^{t L_\mu}\| \|h\|_{L^2(\sqrt{\omega}\dd  \omega)}
    \end{split}
\end{equation}
where we commuted $S_t=e^{t L_\mu}$ with $L_\mu$ for the second term, and for the first term we used the  contractivity of $S_t$.
%contraction from Corollary \ref{cor: decay_contr}.  

From the spectral theorem we have then 
$$\| L_\mu e^{tL_\mu} \| = \sup_{\lambda \in \sigma(L_\mu)} |\lambda e^{t \lambda}| \leq \sup_{\lambda \geq 0} \lambda e^{-t \lambda} = (et)^{-1} \to  0, $$
since $\sigma(L_\mu) \subset (-\infty, 0]$ and since $\lambda e^{-t \lambda}$ is maximised at $\lambda=t^{-1}$. Thus since $\varepsilon > 0$ is arbitrary, we conclude that 
$$\|S_t (I- \Pi)f\|_{L^2(\sqrt{\omega}\dd \omega)}  \to 0 $$
and so $S_t f$ converges strongly to $\Pi f$. 
\end{proof}

\begin{remark}However so far this approach does not give any quantitative or even improved qualitative decay information on the rate of convergence, even though it already indicates stability of the Rayleigh-Jeans on the whole space.
Better decay requires additional knowledge on the spectrum. As we will prove in the next Section the linear operator possesses in fact a spectral gap which gives an exponentially fast return to equilibrium in that functional space. 
\end{remark}

\subsubsection{Determining $\operatorname{Ker}(L_\mu)$}
We finish this section by providing a proof of the characterisation of the collisional invariants, that is determining the Kernel of $L_\mu$. 

Assume that $\omega(k) = |\textbf{k}|^2, d=3$ and look at the problem before considering radial variables and before parametrising. We include a proof of the fact that $\operatorname{Ker}(L_\mu)$ is spanned by $1, \textbf{k}, \omega(\textbf{k})$, $\textbf{k} \in \mathbb{R}^3$, while restricted in $L^2$ with $\mu>0$ it is spanned by $\widetilde{n}_\mu$. 

Since we are going to work in the $k$-variables (without parametrising), we will use the notation $\widetilde L_\mu $, $\widetilde f$ valued at $\textbf{k}$ (and not $\omega$). 
After symmetrising, we remind that  the Dirichlet form is 
\begin{align*}D_{\mu}(\widetilde g)= -\langle L_\mu \widetilde g, \widetilde g\rangle_{L^2} & =\frac{1}{4} \iiiint \prod_{i=1}^3 \widetilde n_\mu (k_i)\times \\
&\times \left[\frac{\widetilde g(k_1)}{\widetilde n_\mu(k_1)} + \frac{\widetilde g(k_2)}{\widetilde n_\mu(k_2)} - \frac{\widetilde g(k)}{\widetilde n_\mu(k)} - \frac{\widetilde g(k_3)}{\widetilde n_\mu(k_3)}\right]^2 \delta_\Omega \delta_\Sigma\  \dd k \dd k_1 \dd k_2 \dd k_3. 
\end{align*}

\begin{proposition}[Kernel of the linearised operator] \label{Prop:Kernel_L}
Let $\mu>0$, $\widetilde{n}_\mu(\textbf{k}) = (|\textbf{k}|^2 + \mu)^{-1}$ and $\widetilde{g} \in L^2(\mathbb{R}^3)$. Then $\widetilde{g} \in \operatorname{Ker}(L_\mu)$ if and only if $\Phi (\textbf{k}) := \widetilde{g}(\textbf{k})/\widetilde{n}_\mu(\textbf{k})$ is a collisional invariant in the sense that, on the resonant manifold: 
    $$\Phi(\textbf{k}_1) + \Phi(\textbf{k}_2) - \Phi(\textbf{k}) - \Phi(\textbf{k}_3)=0, $$
    if and only if  
    $$
    \widetilde{g}(\textbf{k})= \widetilde{n}_\mu(\textbf{k}) ( \alpha \textbf{k} + \gamma |\textbf{k}|^2 ) \in L^2(\mathbb{R}^3)$$ 
    for a.e. $\textbf{k} \in \mathbb{R}^3$, and for some $\alpha, \gamma \in \mathbb{R}$ and $\beta \in \mathbb{R}^3$.
    
    In particular, the condition
    $$\widetilde{n}_\mu(\textbf{k}) (\alpha+\beta\cdot \textbf{k} +\gamma|k|^2) \in L^2(\mathbb{R}^3) $$ forces
    $\beta=0$, $ \gamma=0$. Thus 
    $$ \operatorname{Ker}_{L^2(\mathbb{R}^3)}(L_\mu)=\operatorname{Span}\{\widetilde{n}_\mu\}.
    $$
\end{proposition}

\begin{proof} $\widetilde g \in \operatorname{Ker}(L_\mu)$ if and only if $\langle \widetilde L_\mu \widetilde g,\widetilde g\rangle_{L^2(\mathbb{R}^3)}=0$. This is equivalent to
$$ \Phi(\textbf{k}_1) + \Phi(\textbf{k}_2) - \Phi(\textbf{k}) - \Phi(\textbf{k}_3) =0, \quad \text{ where }\ \Phi(\textbf{k}_i) := \frac{\widetilde g(\textbf{k}_i)}{\widetilde n_\mu(\textbf{k}_i)}, $$ 
for almost every resonant quadruple. Notice also that
$ \Phi=(|\textbf{k}|^2+\mu) \widetilde g \in L^2_{\text{loc}}(\mathbb{R}^3),$ 
since $\widetilde g\in L^2(\mathbb{R}^3)$ and $(|\textbf{k}|^2+\mu)$ is bounded on every compact subset of $\mathbb{R}^3$.

From this, we are going to determine the form of $\Phi$.

We have from the resonant conditions that $$\textbf{k}_1+ \textbf{k}_2 -\textbf{k}-\textbf{k}_3=0, \qquad |\textbf{k}_1|^2+ |\textbf{k}_2|^2 -|\textbf{k}|^2-|\textbf{k}_3|^2=0.$$
Let us write $\textbf{k}_1-\textbf{k}:=a \in \mathbb{R}^3$ and $\textbf{k}_2-\textbf{k}:=b \in \mathbb{R}^3$. Then the energy conservation yields
$$ |\textbf{k}+a|^2+ |\textbf{k}+b|^2 -|\textbf{k}|^2-|\textbf{k}+a+b|^2=0 \quad \Longrightarrow \quad \langle a,b\rangle =0.$$
A collision invariant $\Phi$ satisfies 
$$ F(a,b) := \Phi(\textbf{k}+a)+ \Phi(\textbf{k}+b) - \Phi(\textbf{k}) - \Phi(\textbf{k}+a+b)=0 \quad \text{ whenever} \quad a\perp b.$$  
We first assume that $\Phi\in C^2(\mathbb{R}^3)$ and we differentiate in $a$ and then in $b$. Say that $a=s \hat{e_i}$ and $b=t \hat{e_j}$ for two orthogonal directions $\hat{e_i}, \hat{e_j}$. We first get from our functional identity that
$$\partial_s F(s,t)\vert_{s=0} = \Big[  \nabla \Phi(\textbf{k}+s \hat{e_i}) \cdot \hat{e_i} - \nabla \Phi (\textbf{k}+s \hat{e_i} +t \hat{e_j} ) \cdot \hat{e_i} \Big]\vert_{s=0} =0. 
$$
We then differentiate in $t$ where the first term is zero and we get 
$$
\partial_{t} \partial_{s} F(s,t)\vert_{s=t=0} = \hat{e_j}^T \nabla^2 \Phi (\textbf{k} +s \hat{e_i} +t \hat{e_j} ) \cdot \hat{e_i} \vert_{s=t=0} =0 \quad \Longrightarrow  \ \hat{e_j}^T \nabla^2 \Phi (\textbf{k}) \cdot \hat{e_i} =0.
$$
In other words, all the mixed directions vanish (all the off-diagonal terms of the Hessian are zero)  since $\partial_{ij} \Phi( \textbf{k} ) = [\nabla^2 \Phi ( \textbf{k} )]_{ij}=0$ for all $i\neq j$.

Moreover, if we apply the same argument for the vectors $u=\frac{e_i+e_j}{2},v=\frac{e_i-e_j}{2}$, for $i\neq j$ (which are vertical to each other and so the argument applies), we see that $\frac{1}{2}(\partial_{ii}\Phi - \partial_{jj}\Phi) = u^T \nabla^2\Phi \cdot v =0$ and thus the diagonal entries of the Hessian are all equal to each other. 

Thus the Hessian is a multiple of the identity: 
$$\nabla^2 \Phi (\textbf{k}) = \lambda(\textbf{k}) \operatorname{Id} \quad \text{ for some function } \lambda(\textbf{k}). $$
Now $\lambda(\textbf{k})=\lambda$ has to be a constant: Indeed if we differentiate $\Phi$ once more in a direction $\hat{e_i}$ for $i\neq j$ we get on the one hand
$$ \partial_{ijj}\Phi = \partial_{j}\big(\partial_{ij}\Phi \big) =0 $$
and on the other hand that $\partial_{ijj}\Phi  = \partial_i( \partial_{jj}\Phi)  =\partial_i \lambda(k)$. So $\partial_i \lambda(k)=0$ and $i, j$ were arbitrary. Thus 
$$\nabla^2 \Phi (\textbf{k}) = \lambda \operatorname{Id} \quad \text{ for some } \lambda. $$

Thus $$\Phi(\textbf{k}) = \alpha + \beta  \cdot \textbf{k} + \gamma |\textbf{k}|^2,\quad \text{ for }\quad   \alpha, \gamma \in \mathbb{R},\ \beta \in \mathbb{R}^3 $$
and finally $$\widetilde g(\textbf{k})  = \widetilde n_\mu(\textbf{k}) \Big( \alpha + \beta \cdot \textbf{k} + \gamma |\textbf{k}|^2 \Big). $$

We now remove the regularity assumption on $\Phi$.  As noted above, $\Phi \in L^2_{\text{loc}}(\mathbb{R}^3)$ and it satisfies 
$$
\Phi(\textbf{k}+a)+ \Phi(\textbf{k}+b) - \Phi(\textbf{k}) - \Phi(\textbf{k}+a+b)=0 \quad \text{for a.e. } \textbf{k} \in \mathbb{R}^3 \text{ whenever} \ a\perp b. 
$$

We now consider the standard mollifier $\eta_\varepsilon$ and $\Phi_\varepsilon:= \eta_\varepsilon \ast \Phi$ which approximates $\Phi$ in the $L^2_{\text{loc}}$ topology: $\Phi_\varepsilon \to \Phi$ in $L^2_{\text{loc}}$ as $\varepsilon \downarrow 0$. Then we notice 
\begin{equation} \label{eq:kernel_proof2}
    \begin{split}
        & \Phi_\varepsilon(\textbf{k}+a)  + \Phi_\varepsilon(\textbf{k}) + \Phi_\varepsilon(\textbf{k}+b) + \Phi_\varepsilon(\textbf{k}+a+b)  = \\
  & \int_{\mathbb{R}^3} \eta_\varepsilon(y) [\Phi(\textbf{k}-y+a)+ \Phi(\textbf{k}-y+b) - \Phi(\textbf{k}-y) - \Phi(\textbf{k}-y+a+b)] \dd y=0,
    \end{split}
\end{equation}
since the integrand is $0$ for $\textbf{k}'=\textbf{k}-y \in \mathbb{R}^3$ whenever $a\perp b$. Also since $\Phi_\varepsilon$ is smooth, the left-hand side is continuous in $(\textbf{k},a,b)$ on $a\perp b$, and so the identity \eqref{eq:kernel_proof2} extends to holding from a.e. $(\textbf{k},a,b)$ to every such triple. 
Apply now the previous argument to the smooth $\Phi_\varepsilon$ to get that 
$$ \Phi_\varepsilon = \alpha_{\varepsilon} + \beta_{\varepsilon} \cdot \textbf{k} + \gamma_{\varepsilon} |\textbf{k}|^2, \quad \text{ for } \quad \alpha_{\varepsilon},  \gamma_{\varepsilon}\in \mathbb{R},\ \beta_{\varepsilon}\in \mathbb{R}^3. $$
Since the subspace
 $\operatorname{Span}\{1, k_1, k_2, k_3, |\textbf{k}|^2\}$  is finite dimensional it is closed in $L^2_{\text{loc}}$. Thus any limit is in that space and so we conclude. 
 
The other direction is automatic: if $\Phi(\textbf{k})$ is a linear combination of $1, \textbf{k}, |\textbf{k}|^2$, then it is clearly zero, due to the resonance conditions.  
Therefore we have concluded the first part of the statement that indeed 
$ \operatorname{Ker}_{L^2(\mathbb{R}^3)}(\widetilde L_\mu) = \widetilde n_\mu (\textbf{k}) \operatorname{Span}\{1, \textbf{k}, |\textbf{k}|^2\} \cap L^2(\mathbb{R}^3).$ 

But now notice that since at infinity $\widetilde n_\mu(\textbf{k}) \sim |\textbf{k}|^{-2}$, $k_i \widetilde n_\mu(\textbf{k})$ and $|\textbf{k}|^2 \widetilde n_\mu(\textbf{k}) \notin L^2(\mathbb{R}^3)$. So finally we must have that indeed $\operatorname{Ker}_{L^2(\mathbb{R}^3)}(\widetilde L_\mu)=\operatorname{Span}\{\widetilde n_\mu\}$. 
\end{proof}

\section{Nonlinear Local Well-Posedness}\label{sec:LWP}

We first prove a Proposition that controls, in the $L^2(\sqrt{\omega} \dd \omega)$ topology, both cubic and quadratic nonlinearities that appear in the equation when we consider solutions of the form $(\text{RJ} +\text{RJ}\times \text{perturbation})$.  The proof of that Proposition requires the Raleigh Jeans $n$ to be bounded. Therefore, we fix now the chemical potential $\mu>0$, so to work with non-singular Rayleigh-Jeans equilibria.  

As a consequence, we eventually provide a well-posedness result in the same space, local in time, for the full nonlinear equation, via a fixed point argument.
%Let us cite the recent articles \cite{IoakeimLeger26,IoakeimLeger25} where the authors also deal with local well-posedness for such equations, and in the former article they cover cases of power law initial data but not quite Rayleigh-Jeans yet. 

We remind that the evolution of $f(t,\omega)=(\omega+\mu)^{-1}(1+g(t,\omega))$ from \eqref{eq:NL evolution1} is written as follows: 
\begin{equation}\label{eq:NL evolution3}
\begin{split}
    \partial_t g(t,\omega) &= -n_0^{-1}(\omega)
     \iint_{D(\omega_0)} \frac{\text{min}(\sqrt{\omega_0},\sqrt{\omega_1},\sqrt{\omega_2},\sqrt{\omega_1+\omega_2-\omega_0})}{\sqrt{\omega_0}} \times \\ 
     & \times  \left( f f_2 f_3 + f f_1 f_3 - f_1 f_2 f - f_1 f_2 f_3  \right) \dd \omega_1 \dd \omega_2   \\
 &= L_\mu g +  \Gamma_\mu(g,g) + Q_\mu(g,g,g)
 \end{split}
\end{equation}
with the usual notation $f_3:= f(\omega_1+ \omega_2-\omega)$ and where 
\begin{align*}
L_\mu g = \iint_{D(\omega_0)} \frac{\min (\sqrt{\omega_1}, \sqrt{\omega_2}, \sqrt{\omega_3}, \sqrt{\omega_0})}{ \sqrt{\omega_0}} n_1n_2n_3 \Big[-\frac{g_0}{n_0}  - \frac{g_3}{n_3} + \frac{g_1}{n_1} + \frac{g_2}{n_2}   
\Big] \dd \omega_1 \dd \omega_2, 
\end{align*} 
$ \Gamma_\mu (g,g)$ contains the quadratic nonlinearities: 
\begin{equation}\label{eq: def of Gamma}
\begin{split} 
\Gamma_\mu(g,g) &= \iint_{\substack{\omega_1, \omega_2 \in (0,\infty)\\ \omega_1+\omega_2>\omega_0}}  \frac{\text{min}(\sqrt{\omega_0},\sqrt{\omega_1},\sqrt{\omega_2},\sqrt{\omega_1+\omega_2-\omega_0})}{\sqrt{\omega_0}} 
\times \\
& \hskip 6cm  \times n(\omega_1)n(\omega_2)n(\omega_1+\omega_2-\omega)\times \\
&\times \Big[ -g_0g_3 ( n_2^{-1} + n_1^{-1} ) + 
    g_0g_1 ( n_3^{-1} - n_2^{-1}) + g_0g_2 (n_3^{-1} - n_1^{-1}) \\ 
    &\hspace{0cm} + 
    g_3g_1 (n_1^{-1} - n_3^{-1}) + g_3g_2 (n_2^{-1} - n_3^{-1}) + g_1g_2 (n_2^{-1} + n_1^{-1}  )
\Big] \dd \omega_1 \dd\omega_2, 
\end{split}
\end{equation} 
 and $Q_\mu(g,g,g)$ contains the cubic nonlinearities: 
\begin{equation}\label{eq: def of Q}
\begin{split} 
 Q_\mu(g,g,g)
 = 
 &\iint_{\substack{\omega_1, \omega_2 \in (0,\infty)\\ \omega_1+\omega_2>\omega_0}}  \frac{\text{min}(\sqrt{\omega_0},\sqrt{\omega_1},\sqrt{\omega_2},\sqrt{\omega_1+\omega_2-\omega_0})}{\sqrt{\omega_0}} 
   n_1n_2n_3\times \\
  &\hspace{0cm}\times \Big[   n_3^{-1} g_0g_1g_2  +  n_0^{-1}g_1g_2g_3 -  n_1^{-1}g_0g_2g_3  -  n_2^{-1}g_0g_1g_3 \Big] \dd \omega_1 \dd \omega_2.
 \end{split}
\end{equation}  

\begin{proposition}\label{prop: control NL} Let $\mu>0$.
    It holds that 
    $$\|\Gamma_\mu (g,h)\|_{L^2(\sqrt{\omega}\dd\omega)} \leq C_\mu \|g\|_{L^2(\sqrt{\omega}\dd\omega)}\|h\|_{L^2(\sqrt{\omega}\dd\omega)}$$
    and 
    $$\|Q_\mu (g,h, \ell)\|_{L^2(\sqrt{\omega}\dd\omega)} \leq C_\mu \|g\|_{L^2(\sqrt{\omega}\dd\omega)} \|h\|_{L^2(\sqrt{\omega}\dd\omega)} \|\ell\|_{L^2(\sqrt{\omega}\dd\omega)},  
    $$
    for some finite constant $C_\mu>0$. 
    \end{proposition}

\begin{remark}[On the dependency of the constant $C_\mu$ on $\mu$]
Note that the estimate of $C_\mu $ above depends on $\mu$ and in particular it blows up as $\mu \to 0$. Indeed using the same dilation unitary operator as in the proof of Prop. \ref{prop:noncompactness}, $(U_\lambda g)(\omega) := \lambda^{-3/4} g(\lambda^{-1}\omega)$, we see that for example for the quadratic nonlinearities: $U_\lambda^{-1} \Gamma_{\mu} (U_\lambda g,U_\lambda g) = \lambda^{-3/4}\Gamma_{\tfrac{\mu}{\lambda}} (g,g) $. Indeed briefly, $(U_\lambda^{-1}g)(\omega)=\lambda^{3/4}g(\lambda\omega)$ and setting $ \omega_i=\lambda x_i$ in the integral, the measure contributes $\lambda^2$, the two RJ  factors contribute $\lambda^{-2}$, since $$ n_\mu(\lambda x)=\lambda^{-1}n_{\mu/\lambda}(x),  $$ the two factors $U_\lambda g$ contribute $\lambda^{-3/2}$, while the kernel is 0-homogeneous, and $U_\lambda^{-1}$ contributes $\lambda^{3/4}$. Altogether, $$ U_\lambda^{-1}\Gamma_\mu(U_\lambda g,U_\lambda g) = \lambda^{3/4+2-2-3/2} \Gamma_{\mu/\lambda}(g,g) = \lambda^{-3/4}\Gamma_{\mu/\lambda}(g,g). $$
Then $$\| U_\lambda^{-1} \Gamma_{\mu} (U_\lambda g,U_\lambda g)\|_{L^2(\sqrt{\omega})} = \| \Gamma_{\mu} (U_\lambda g,U_\lambda g) \|_{L^2(\sqrt{\omega})}  = \lambda^{-3/4} \|\Gamma_{\tfrac{\mu}{\lambda}} ( g,g) \|_{L^2(\sqrt{\omega})}. $$
For $\mu=\lambda$, we obtain the bound $C_1\mu^{-3/4}$. For the cubic term, a similar scaling in $\mu$ holds with power $-3/2$ instead of $-3/4$. Thus indeed our hypothesis $\mu>0$ is not just a technicality. 
%shows that our bounds of the nonlinear terms blow up as $\mu \to 0$. 
\end{remark}
    \begin{proof}
        \noindent
\underline{\emph{We start with the cubic nonlinearities}} which are more difficult. 
        We first notice that the terms for which the nonlinearity is of the form $$n_i^{-1}g_0g_kg_\ell \quad \text{ for }\ i, k, \ell \in \{1,2,3\},$$
        i.e. have $g_0$ as a factor, can be treated as follows. For example for the first term we write: 
        \begin{align*}
            &\left\| \iint_{\substack{\omega_1, \omega_2>0 \\ \omega_1+\omega_2>\omega_0}}  \frac{\text{min}(\sqrt{\omega_0},\sqrt{\omega_1},\sqrt{\omega_2},\sqrt{\omega_1+\omega_2-\omega_0})}{\sqrt{\omega_0}} 
   n_1n_2 g_0g_1g_2 \right\|_{L^2(\sqrt{\omega_0})}^2 \\
   & \leq \int_0^\infty \sqrt{\omega_0} |g_0|^2 \left\vert 
   \iint_{\substack{\omega_1, \omega_2>0 \\ \omega_1+\omega_2>\omega_0}}  \frac{\text{min}(\sqrt{\omega_0},\sqrt{\omega_1},\sqrt{\omega_2},\sqrt{\omega_3})}{\sqrt{\omega_0}} n_1n_2 g_1g_2 \dd \omega_1 \dd \omega_2
   \right\vert^2 d \omega  \\
   & \leq \|g\|_{L^2(\sqrt{\omega} d\omega)}^2\|g\|_{L^2(\sqrt{\omega} d\omega)}^4 = \|g\|_{L^2(\sqrt{\omega} d\omega)}^6, \end{align*}
        after bounding the minimum factor by $1$, multiplying and dividing by $\omega^{1/4}$ and applying H\"{o}lder's inequality in the inner double integral. The remaining two terms that look like this, are quite similar and thus omitted. 
        
Therefore,  the only delicate term in the cubic nonlinearity is the one with the nonlinearity of the form $n_0^{-1} g_1g_2g_3$. We call this integral $T_{\text{del}}(g,g,g)$.  
To bound this term, the idea is to proceed by a dyadic decomposition of
our integrals. We split the integral into annuli for $j \geq 0$,
so that the integral can be decomposed into 
\begin{align*} 
 I_j &:= \{ \omega >0: 2^{j} < \omega \leq 2^{j+1}\}, \quad \text{for } j\geq 0,  \\
 I_{-1}&:= \{ \omega >0: \omega \in (0,1]\}.
\end{align*}
We then control each integral on these annuli. 

 We start by writing $g = g 1_{I_{-1}} + \sum_{j \geq 0} g 1_{I_j}$ and assuming that $\omega_i \in I_{j_i}$ for each $i=-1,1,2,3,0$. We first going to show the estimate for $\min(j_0, j_1,j_2,j_3) \geq 0$, that is when all $\omega_{i}$'s are larger than $1$. 
 The following concerns bounding the localised-for-each-frequency delicate term  
        $$ 1_{I_{j_0}} T_{\text{del}}(g_{j_1},g_{j_2},g_{j_3}), \quad 
        \text{ where }\  g_{j_i} = g  1_{I_{j_i}}, 
        $$
        since the original term is recovered by gluing the localised terms together:
        \begin{align*}
             \|T_{\text{del}}(g,g,g)\|_{L^2(\sqrt{\omega})}& = 
            \left[\int_0^\infty \sqrt{\omega} |T_{\text{del}}|^2 \dd \omega \right]^{1/2}\\
            &= 
            \left[ \sum_{j_0\geq 0} \int_{I_{j_0}} \sqrt{\omega_0} \left\vert  
            \sum_{j_1, j_2, j_3} T_{\text{del}}(g_{j_1},g_{j_2},g_{j_3})
            \right\vert^2 \dd \omega_0 \right]^{1/2} \\
            &  = \left[\sum_{j_0\geq 0}
            \left\| \sum_{j_1, j_2, j_3} T_{\text{del}}(g_{j_1},g_{j_2},g_{j_3})
            \right\|_{L^2(I_{j_0};\sqrt{\omega}\dd \omega)}^2   \right]^{1/2} \\
            & \leq 
            \left[\sum_{j_0\geq 0} \left( 
            \sum_{j_1,j_2,j_3}
            \|T_{\text{del}}(g_{j_1},g_{j_2},g_{j_3})  \|_{L^2(I_{j_0};\sqrt{\omega}\dd \omega)} \right)^2 
            \right]^{1/2}
        \end{align*}
        
        where the last step is just the triangle inequality. 
        Above we may also assume that $j_1\geq j_2 \geq j_3$. We can do that up to a factor $6$ of all the possible permutations, since $\sum_{j_1,j_2,j_3} = \sum_{\text{permutations}} \sum_{j_{\sigma(1)} \geq j_{\sigma(2)} \geq j_{\sigma(3)}}$.  

        We estimate now the integral kernel:
        \begin{align}
            \sigma_{0,1,2,3}:&= \frac{\text{min}(\sqrt{\omega_0},\sqrt{\omega_1},\sqrt{\omega_2},\sqrt{\omega_3})}{\sqrt{\omega_0}} 
   \frac{n_1n_2n_3}{n_0}\notag \\
   &\lesssim \min \left(1,
   \sqrt{
   \frac{2^{\min(j_1,j_2,j_3)}}{2^{j_0}} 
   } \right) 2^{j_0 - j_1-j_2-j_3}  = \min \left(1,
   \sqrt{
   \frac{2^{j_3}}{2^{j_0}} 
   } \right) 2^{j_0 - j_1-j_2-j_3}
        \end{align}
        because we have assumed that $j_3$ is the minimum scale among the $j_1,j_2,j_3$. So then 
        \begin{equation} \label{eq:cubic term_dydadic}
            \begin{split}
                &\|T_{\text{del}}(g_{j_1},g_{j_2},g_{j_3})  \|_{L^2(I_{j_0};\sqrt{\omega}\dd \omega)}   = 
                \left[ \int_{I_{j_0}}\sqrt{\omega_0} \left\vert
                \iint \sigma_{0,1,2,3}  \ g_{j_1}g_{j_2}g_{j_3} \dd \omega_1 \dd \omega_2 
                \right\vert^2 \dd \omega_0\right]^{1/2}\\
                & \lesssim \min \left(1, \sqrt{
     \frac{2^{j_3}}{2^{j_0}} 
   } \right) 2^{j_0-j_1-j_2-j_3} 2^{j_0/4} 
   \left[ \int_{I_{j_0}}  \left\vert 
   \iint g_{j_1}g_{j_2}g_{j_3}\dd \omega_1\dd \omega_2
   \right\vert^2
   \right]^{1/2} \\
   &  = \min \left(1, 2^{(j_3-j_0)/2} \right) 2^{5 j_0/4-j_1-j_2-j_3}
   \| g_{j_1} \ast g_{j_2} \ast \overline{g_{j_3}} \|_{L^2(\dd \omega)}
            \end{split}
        \end{equation}
        where the factor $2^{j_0/4}$ comes due to the weight $\sqrt{\omega_0}$. Here we extend each $g_{j_3}$ by zero to the whole real line and denote by
$\overline{g_{j_3}}(z):=g_{j_3}(-z)$ its reflection. With this convention,
the resonance relation $\omega_3=\omega_1+\omega_2-\omega_0$ allows us to write the integral as a convolution on $\mathbb{R}$. This is necessary as the convolution makes $g_{j_3}$ to be valued at $-\omega_3$. 
        
        We continue by applying Young's Inequality to the convolution term in \eqref{eq:cubic term_dydadic}: 
        \begin{equation}\label{eq:Young_Ineq_convolution}
         \| g_{j_1} \ast g_{j_2} \ast \overline{g_{j_3}} \|_{L^2(\dd \omega)} \leq \|g_{j_1} \|_{L^2(\dd \omega)} \|g_{j_2} \|_{L^1(\dd \omega)} \| g_{j_3} \|_{L^1(\dd \omega)}, 
         \end{equation}
        where we choose to place the $L^2$ norm in the term localised in the highest frequency (here assumed to be $j_1$) so that we extract more decay. To get the estimate in our weighted topology $L^2(\sqrt{\omega}\dd \omega)$, we compute explicitly 
        \begin{align} \label{eq:estimL2_1}
            & \|g_{j_i}\|_{L^2(\dd \omega)} \lesssim 2^{-j_i/4} \|g_{j_i}\|_{L^2(\sqrt{\omega}\dd\omega)}\\
            \label{eq:estimL2_2}
            & \|g_{j_i}\|_{L^1(\dd \omega)} \lesssim 
            2^{j_i/4} \|g_{j_i}\|_{L^2(\sqrt{\omega}\dd\omega)}.
        \end{align}
        Inserting these estimates in \eqref{eq:cubic term_dydadic} we get  
        \begin{equation}
            \begin{split}
                 \|T_{\text{del}}(g_{j_1},g_{j_2},g_{j_3})  \|_{L^2(I_{j_0};\sqrt{\omega}\dd \omega)}  
                 \lesssim  
                &\min \left(1, 2^{(j_3-j_0)/2} \right) 2^{5 j_0/4 - 5j_1/4 - 3j_2/4 - 3j_3/4} \times  \\ &\times \|g_{j_1}\|_{L^2(\sqrt{\omega}\dd\omega)}
                \| g_{j_2}\|_{L^2(\sqrt{\omega}\dd\omega)}
                \| g_{j_3}\|_{L^2(\sqrt{\omega}\dd\omega)}.
            \end{split}
        \end{equation}
        Therefore, altogether we get 
        \begin{equation}
            \begin{split}
                        &\|T_{\text{del}}(g,g,g)\|_{L^2(\sqrt{\omega})} 
                        \lesssim  
                       \Bigg[ \sum_{j_0\geq 0}
                       \Bigg( \sum_{j_1\geq j_2 \geq j_3}   
                       \min \left(1, 2^{(j_3-j_0)/2} \right) \times \\ 
                       &  \times 
                    2^{5 j_0/4 - 5j_1/4 - 3j_2/4 - 3j_3/4}\|g_{j_1}\|_{L^2(\sqrt{\omega}\dd\omega)}
                \|g_{j_2}\|_{L^2(\sqrt{\omega}\dd\omega)}
                \| g_{j_3}\|_{L^2(\sqrt{\omega}\dd\omega)} \Bigg)^2 \Bigg]^{1/2}.
            \end{split}
        \end{equation}
        
        Now we split into two possible cases on whether $j_0$ is larger or smaller to $j_1$. \\

        \noindent
        \underline{Case 1: If $j_0 \geq j_1 - \text{Const.}$:} The resonant relation forces then that 
        $j_0 \sim j_1$. Indeed we have 
         $2^{j_0}\sim \omega_0  = \omega_1+\omega_2 - \omega_3 \leq \omega_1+\omega_2 \leq 2^{j_1+1}$ since $j_1$ is assumed to be the largest scale. So if 
$j_0 \geq j_1 - \text{Const.}$, it has to be the case that $j_0 \sim j_1$. In other words, $\omega_1$ and $\omega_0$ live on comparable dyadic scales, in that scenario.
%j_1-constant \leq j_0 \leq j_1+1
Inserting that in our main term, we don't have the sum over $j_1$ anymore (since $j_1$ and $j_0$ differ by at most a fixed constant). Then 
\begin{equation}
    \begin{split}
      &\|T_{\text{del}}(g,g,g)\|_{L^2(\sqrt{\omega})} 
        \lesssim \\ & 
        \left[\sum_{j_0\geq 0}
                       \left( \sum_{j_0\geq j_2 \geq j_3} 
                       2^{(j_3-j_0)/2}
                       2^{-3j_2/4 - 3j_3/4}
\|g_{j_0}\|_{L^2(\sqrt{\omega}\dd\omega)}
 \| g_{j_2}\|_{L^2(\sqrt{\omega}\dd\omega)}
\| g_{j_3}\|_{L^2(\sqrt{\omega}\dd\omega)}\right)^2 \right]^{1/2}\\
& = \left[\sum_{j_0\geq 0}
2^{-j_0}
\|g_{j_0}\|_{L^2(\sqrt{\omega}\dd\omega)}^2 
              \left( \sum_{j_0\geq j_2}       
 2^{-3j_2/4} \| g_{j_2} \|_{L^2(\sqrt{\omega}\dd\omega)}
 \sum_{j_3 \leq j_2 } 2^{-j_3/4} 
\| g_{j_3}\|_{L^2(\sqrt{\omega}\dd\omega)}\right)^2 \right]^{1/2}
\\
    & \lesssim \left[\sum_{j_0\geq 0} 2^{-j_0}\|g_{j_0}\|_{L^2(\sqrt{\omega}\dd\omega)}^2
                \right]^{1/2} \|g\|_{L^2(\sqrt{\omega}\dd\omega)}^2 \lesssim \|g\|_{L^2(\sqrt{\omega}\dd\omega)}^3.
           \end{split}
        \end{equation}
where in the last inequality we applied Cauchy-Schwarz. 

        \noindent
        \underline{Case 2: If $j_0 < j_1 - \text{Const.}$:} In this case the resonant condition forces $j_1\sim j_2 \sim j_3$. Indeed the resonances give that $\omega_1+ \omega_2=\omega_0+\omega_3$ and the left-hand side is dominated by $\omega_1$ since $j_1$ is the largest scale and $\omega_1+\omega_2 \sim 2^{j_1}$. In order to balance the large $\omega_1$ on the left, the right-side must contain a comparable large term. But  given that $\omega_0< \omega_1$, it must be that $\omega_3 \sim \omega_1$ or equivalently that $j_3 \sim j_1$. Now, since we have assumed that $j_1\geq j_2\geq j_3$ it has to be that all  frequencies are in comparable dyadic scales and thus $j_1\sim j_2 \sim j_3$. 

        Inserting that in our main term, we just have two main sums, since $j_1,j_2,j_3$ differ by at most constants. So:
        \begin{equation}
    \begin{split}
      &\|T_{\text{del}}(g,g,g)\|_{L^2(\sqrt{\omega})} 
        \lesssim \\ & 
        \left[ \sum_{j_0\geq 0}
                       \left( \sum_{j_1> j_0 + \text{Const.}}   
                        2^{5 j_0/4 - 11 j_1/4}
            \|g_{j_1}\|_{L^2(\sqrt{\omega}\dd\omega)}^3
                 \right)^2 \right]^{1/2} \\
                 & \lesssim 
                 \sum_{j_1\geq 0} 
                 \left[ \sum_{j_0} \left(1_{j_1>j_0+\text{Const.}} 2^{5 j_0/4 - 11 j_1/4}\|g_{j_1}\|_{L^2(\sqrt{\omega}\dd\omega)}^3\right)^2
                 \right]^{1/2}\\
                 & = 
                 \sum_{j_1\geq 0} 2^{- 11 j_1/4} \|g_{j_1}\|_{L^2(\sqrt{\omega}\dd\omega)}^3 
                 \left( \sum_{j_0< j_1 - \text{Const.}} 2^{5 j_0/2}  \right)^{1/2}\\
                 & \lesssim \sum_{j_1\geq 0} 2^{- 11 j_1/4} \|g_{j_1}\|_{L^2(\sqrt{\omega}\dd\omega)}^3 2^{5j_1/4}
                 \\
                 & = 
                 \sum_{j_1\geq 0} 2^{- 3 j_1/2} \|g_{j_1}\|_{L^2(\sqrt{\omega}\dd\omega)}^3  \lesssim \|g\|_{L^2(\sqrt{\omega}\dd\omega)}^3 
        \end{split}
        \end{equation}
        where the minimum-factor in the beginning is just $1$, for the second inequality we applied triangle inequality in the $\ell^2_{j_0}$-norm, and from the fourth to fifth line we bounded the geometric series by $2^{5(j_1 -\text{Const.})/2}\lesssim 2^{5 j_1/2}$.
        
For the remaining cases regarding the ordering of the scales $j_1, j_2, j_3$, i.e. when $j_3$ or $j_2 = \max(j_1,j_2, j_3)$, the argument is very similar and thus omitted. What is important in each case is to share the $L^2-L^1-L^1$ norm in the Young's inequality for convolutions, in \eqref{eq:Young_Ineq_convolution},  by placing the $L^2$ norm to the highest scales assumed. 

Finally notice that so far we treated all the scenaria as long as all the frequencies are larger than $1$. This suffices since the low-frequency block $I_{-1}$ is harmless for $\mu>0$: Indeed, we can use that on bounded frequencies, $n_i=(\omega+\mu)^{-1}$ and $n_i^{-1}$ are all bounded and estimates like \eqref{eq:estimL2_2} are still true: 
$$
\|g_{-1}\|_{L^1(\dd\omega)} \leq \left(\int_0^1 |g(\omega)|^2\sqrt{\omega} \dd \omega\right)^{1/2}
\left(\int_0^1 \omega^{-1/2} \dd\omega\right)^{1/2}
\lesssim
\|g_{-1}\|_{L^2(\sqrt{\omega}\dd\omega)}.
$$
In the Young convolution estimates above, the $L^2(\dd\omega)$ factor is always chosen to be in the highest frequency term.\\
So, \emph{if the highest frequency is not in the low block $I_{-1}$}, we use
\eqref{eq:estimL2_1} for that highest-frequency factor, 
while the remaining lower frequency factors are placed in $L^1(\dd\omega)$ slot as was done above. \\

\emph{If the highest frequency belongs in $I_{-1}$}, then all input frequencies are bounded: indeed
 if we have $j_1\geq j_2\geq j_3$ and $j_1=-1$ then $j_2= j_3=-1$ as well. Also by the resonance relation $\omega_0= \omega_1+\omega_2-\omega_3<2$, and thus bounded, too.
 In this case the space  $L^2((0,1);\sqrt{\omega}\dd\omega)$ is not a subspace of $L^2((0,1);\dd\omega)$ so we need to adjust the estimate \eqref{eq:estimL2_1} when applying Young's inequality for convolutions.
In fact we have $L^2((0,1);\sqrt{\omega}\dd\omega) \subset L^{\frac{6}{5}}((0,1))$: Indeed by H\"{o}lder's  
\begin{align*}
\int_0^1 |g(\omega)|^{6/5}\dd\omega =\int_0^1 \Big(|g(\omega)|^2\sqrt{\omega} \Big)^{3/5}\omega^{-3/10}\dd\omega & \leq \left(\int_0^1|g(\omega)|^2\sqrt{\omega}\dd\omega\right)^{3/5}
\left(\int_0^1\omega^{-3/4}\dd\omega\right)^{2/5} \lesssim \|g\|_{L^2(\sqrt{\omega})}^{6/5}.
\end{align*} 
In other words $\|g_{-1}\|_{L^{6/5}} \lesssim \|g_{-1}\|_{L^2(\sqrt{\omega})}$. Then in the convolution term we apply Young's inequality\footnote{We remind that the general form of Young's  inequality for 3 convolutions is $\|f \ast g \ast h \|_{L^r} 
\leq 
\|f\|_{L^p} \|g\|_{L^q} \|h\|_{L^s} $ for $\frac{1}{r}+2 = \frac{1}{p} + \frac{1}{q}+ \frac{1}{s}$. Here in the low frequency block, we applied it for $p=q=s=6/5$.},
$$
\|g_{-1}*g_{-1}*\overline{ g_{-1}}\|_{L^2(\dd\omega)}
\lesssim \|g_{-1}\|_{L^{6/5}(\dd\omega)}^3
\lesssim \|g_{-1}\|_{L^2(\sqrt{\omega}\dd\omega)}^3.
$$
Since the output is supported on a bounded interval, the weighted $L^2$ norm is controlled by the unweighted $L^2$ norm.
 %Thus the whole contribution is estimated directly on a bounded set (since the weights are also bounded).  
Thus we conclude that for $\mu>0$:
$$
\|Q_\mu(g,g,g) \|_{L^2(\sqrt{\omega}\dd\omega)} \lesssim \|g\|_{L^2(\sqrt{\omega}\dd\omega)}^3. 
$$ 

\noindent
\underline{\emph{We now treat the quadratic nonlinearities}.} The terms containing nonlinearities of the form 
$g_0g_i n_1n_2 n_3 (n_k^{-1} \pm n_\ell^{-1})$ for $i, k, \ell \in \{1,2,3\}$ can be treated by H\"{o}lder's since $g_0$ factors out. We show that for the first term for example (which contains a sum of $n_k$, $n_\ell$): 
\begin{align*}
& \left\| 
\iint_{\substack{\omega_1, \omega_2 \in (0,\infty)\\ \omega_1+\omega_2>\omega_0}} 
\frac{\text{min}(\sqrt{\omega_0},\sqrt{\omega_1},\sqrt{\omega_2},\sqrt{\omega_3})}{\sqrt{\omega_0}} 
    n_1n_2n_3 g_0 g_3 ( n_2^{-1} + n_1^{-1} ) \dd\omega_1\dd\omega_2
\right\|_{L^2(\sqrt{\omega})} \lesssim \\
& 
\| g_0 \|_{L^2(\sqrt{\omega})} 
\left\| 
\iint_{\substack{\omega_1, \omega_2 \in (0,\infty)\\ \omega_1+\omega_2>\omega_0}}  \frac{\text{min}(\sqrt{\omega_0},\sqrt{\omega_1},\sqrt{\omega_2},\sqrt{\omega_1+\omega_2-\omega_0})}{\sqrt{\omega_0}} 
 g_3  n_1n_3  \dd\omega_1\dd\omega_2
\right\|_{L_{\omega_0}^\infty} + \\& 
 \| g_0 \|_{L^2(\sqrt{\omega})} 
 \left\| 
 \iint_{\substack{\omega_1, \omega_2 \in (0,\infty)\\ \omega_1+\omega_2>\omega_0}}   \frac{\text{min}(\sqrt{\omega_0},\sqrt{\omega_1},\sqrt{\omega_2},\sqrt{\omega_1+\omega_2-\omega_0})}{\sqrt{\omega_0}}  
   g_3  n_2 n_3  \dd\omega_1\dd\omega_2
\right\|_{L_{\omega_0}^\infty}.
\end{align*}
We treat the first one since the second is identical by the symmetry $\omega_1\leftrightarrow \omega_2$.
So in order to bound uniformly-in-$\omega_0$ the first term, set
$$
K_1(\omega_0)
:= \iint_{\omega_1+\omega_2>\omega_0}
\frac{\text{min}(\sqrt{\omega_0},\sqrt{\omega_1},\sqrt{\omega_2},\sqrt{\omega_3})}{\sqrt{\omega_0}} n_1n_3 |g_3| \dd\omega_1 \dd \omega_2,
$$ and we make the change of variables
$\omega_3=\omega_1+\omega_2-\omega_0,$ and $\omega_2=\omega_0+\omega_3-\omega_1.$ Then 
\begin{align*} 
K_1(\omega_0) 
&= \int_0^\infty |g(\omega_3)|n_3
\left(
\int_0^{\omega_0+\omega_3}
\frac{\text{min}(\sqrt{\omega_0},\sqrt{\omega_1},\sqrt{\omega_2},\sqrt{\omega_3})}{\sqrt{\omega_0}} n_1 d\omega_1
\right) d\omega_3
\\
& \lesssim 
\int_0^\infty |g(\omega_3)|n_3 \left( \min\left(1,\sqrt{\frac{\omega_3}{\omega_0}}\right)
\int_0^{\omega_0+\omega_3}\frac{d\omega_1}{\omega_1+\mu} \right) \dd\omega_3
\\ & 
\lesssim_\mu 
\int_0^\infty |g(\omega_3)|n_3 
\left(
\min\left(1,\sqrt{\frac{\omega_3}{\omega_0}}\right)
\log(1+\omega_0+\omega_3)
\right) \dd \omega_3, 
\end{align*}
where we used that the min factor is always less than $\sqrt{\omega_3/\omega_0}$ and than $1$. 
Now uniformly in $\omega_0>0$, it holds that 
$$\min\left(1,\sqrt{\frac{\omega_3}{\omega_0}}\right)
\log(1+\omega_0+\omega_3)
\lesssim
1+\log(1+\omega_3) +\omega_3^{\delta},$$ 
for any fixed $\delta \in (0,1/2)$. 
Indeed, if $\omega_0 \le \omega_3$, this is immediate, while if
$\omega_0>\omega_3$, 
\begin{align*}
\sqrt{\frac{\omega_3}{\omega_0}}
\log(1+\omega_0+\omega_3)& \lesssim \sqrt{\frac{\omega_3}{\omega_0}} (1+\omega_0+\omega_3)^{\delta}\lesssim \sqrt{\frac{\omega_3}{\omega_0}}  (1+\omega_0)^{\delta}\\
&\lesssim \omega_3^{1/2}\omega_0^{-1/2} \omega_0^{\delta} \lesssim \omega_3^{1/2} \omega_3^{-1/2+\delta} \lesssim 1+ \omega_3^\delta. 
\end{align*}

Inserting this to the $K_1(\omega_0)$ term for $\delta=1/4$ and by Cauchy-Schwarz, we have
\begin{align*}
K_1(\omega_0)
&\lesssim_\mu
\int_0^\infty
|g(\omega_3)|
\frac{1+\log(1+\omega_3)+\omega_3^{\delta}}
{\omega_3+\mu}\dd \omega_3\\
&=
\int_0^\infty
|g(\omega_3)|\omega_3^{1/4}
\frac{1+\log(1+\omega_3)+\omega_3^{1/4}}
{(\omega_3+\mu)\omega_3^{1/4}}
\dd\omega_3
\\
&\le
\|g\|_{L^2(\sqrt{\omega}\,d\omega)}
\left(
\int_0^\infty
\frac{
\left(1+\log(1+\omega)+\omega^{1/4}\right)^2
}
{(\omega+\mu)^2\sqrt{\omega}}
\dd \omega
\right)^{1/2}.
\end{align*}
Since $\mu>0$, the last integral is finite. Eventually indeed 
$
\|K_1\|_{L^\infty_{\omega_0}}
\lesssim_\mu
\|g\|_{L^2(\sqrt{\omega}\dd\omega)}.$
Thus we conclude the desired $\|g\|_{L^2(\sqrt{\omega}\dd\omega)}^2$ bound for the first term. Similarly follow the rest terms where $g_0$ factors out.  

The terms containing a difference of $n_k$, $n_\ell$ for example estimating: 
\begin{align*}
& \left\| 
\iint_{\substack{\omega_1, \omega_2 \in (0,\infty)\\ \omega_1+\omega_2>\omega_0}} 
\frac{\text{min}(\sqrt{\omega_0},\sqrt{\omega_1},\sqrt{\omega_2},\sqrt{\omega_3})}{\sqrt{\omega_0}} 
    n_1n_2n_3 g_0 g_1 ( n_3^{-1} - n_2^{-1} ) \dd\omega_1\dd\omega_2
\right\|_{L^2(\sqrt{\omega})}.
\end{align*}
After factoring out $g_0$, the corresponding coefficient is bounded by
\begin{align*}
 M(\omega_0) &:=\int_0^\infty |g(\omega_1)|n_1
\min\left(1,\sqrt{\frac{\omega_1}{\omega_0}}\right)
\left(\int_{(\omega_0-\omega_1)_+}^\infty
\frac{|\omega_1-\omega_0|}{(\omega_2+\mu)(\omega_3+\mu)}\dd\omega_2\right)\dd\omega_1 \\ &=\int_0^\infty |g(\omega_1)|n_1
\min\left(1,\sqrt{\frac{\omega_1}{\omega_0}}\right)
\log\left(1+\frac{|\omega_1-\omega_0|}{\mu}\right)\dd\omega_1, 
\end{align*}
and for every fixed $\delta\in(0,1/2)$,
$$
\min\left(1,\sqrt{\frac{\omega_1}{\omega_0}}\right)
\log\left(1+\frac{|\omega_1-\omega_0|}{\mu}\right)
\lesssim_{\mu,\delta}1+\log(1+\omega_1)+\omega_1^\delta,
$$
uniformly in $\omega_0>0$. Indeed, if $\omega_0 \leq \omega_1$, it follows directly since $|\omega_1-\omega_0|\leq \omega_1$, while if $\omega_0>\omega_1$ and $\omega_0\leq 1$, the left-hand side is uniformly bounded, and if $\omega_0 >\max\{ 1,\omega_1 \}$, one uses
$$
\sqrt{\frac{\omega_1}{\omega_0}}
\log\left(1+\frac{\omega_0-\omega_1}{\mu}\right)
\lesssim_{\mu,\delta}
\omega_1^{1/2}\omega_0^{-1/2+\delta}
\lesssim 1+\omega_1^\delta.
$$
As above take $\delta=1/4$, and apply Cauchy-Schwarz to conclude that $
\|M\|_{L^\infty_{\omega_0}}
\lesssim_\mu \|g\|_{L^2(\sqrt{\omega}\dd\omega)}.
$
The remaining terms of this form are identical and thus ommitted.

We now turn to the quadratic terms not containing $g_0$, where we again employ a dyadic decomposition argument. After triangle inequality each of these terms have integral kernel of the form $$\frac{\text{min}(\sqrt{\omega_0},\sqrt{\omega_1},\sqrt{\omega_2},\sqrt{\omega_1+\omega_2-\omega_0})}{\sqrt{\omega_0}} n_in_j |g_k g_\ell|.$$
We again show the claim for one representative estimate, for instance
\begin{align*}
& T_2(g,g)(\omega_0):=\iint_{\omega_1+\omega_2>\omega_0}
\frac{\text{min}(\sqrt{\omega_0},\sqrt{\omega_1},\sqrt{\omega_2},\sqrt{\omega_1+\omega_2-\omega_0})}{\sqrt{\omega_0}}n_1n_3 g_1g_2 \dd \omega_1\dd \omega_2. 
\end{align*}

We will use the same notation, i.e. $I_j=\{\omega \in (2^j, 2^{j+1}]\}$ and the estimates on the $L^2, L^1$ norms \eqref{eq:estimL2_1} and \eqref{eq:estimL2_2}.

Assume first that $j_1\geq j_2$. The integral kernel is bounded by 
$$\text{integral kernel } \lesssim \min\left(1, 2^{(\min(j_2,j_3) - j_0)/2} \right) 2^{-j_1-j_3}$$ and for the localised term we apply a Young's inequality for convolutions: 
 \begin{align*}
     \left\| 
     \iint_{\omega_1+\omega_2\geq \omega_0}
     g_{j_1}g_{j_2} 1_{\omega_3 \in I_{j_3}} \dd \omega_1\dd\omega_2
     \right\|_{L^2} 
     &\lesssim \|g_{j_1}\|_{L^2} \|g_{j_2}\|_{L^1} 2^{j_3}
     \end{align*}
     so that altogether yield 
     \begin{align*}
         &\|1_{I_{j_0}} T_2(g_{j_1},g_{j_2})\|_{L^2(\sqrt{\omega_0})}  \lesssim 2^{j_0/4} \min\left(1, 2^{(\min(j_2,j_3) - j_0)/2} \right) 2^{-j_1-j_3}\|g_{j_1}\|_{L^2} \|g_{j_2}\|_{L^1} 2^{j_3}\\
         &\hskip 1.5cm \lesssim 
         2^{j_0/4} \min\left(1, 2^{(\min(j_2,j_3) - j_0)/2} \right) 2^{-5j_1/4}2^{j_2/4}\|g_{j_1}\|_{L^2(\sqrt{\omega}\dd\omega)} \|g_{j_2}\|_{L^2(\sqrt{\omega}\dd\omega)}. 
     \end{align*}
     For the full term, we glue the localised terms together: 
\begin{align*}
&\|T_2(g,g) \|_{L^2(\sqrt{\omega}\dd\omega)}\lesssim 
\Bigg[\sum_{j_0\geq 0} \Bigg( \sum_{j_1\geq j_2,j_3} 2^{j_0/4}\min\left(1, 2^{(\min(j_2,j_3) - j_0)/2} \right)\times  \\
& \hskip 4.5cm \times 2^{-5j_1/4}2^{j_2/4}\|g_{j_1}\|_{L^2(\sqrt{\omega}\dd\omega)} \|g_{j_2}\|_{L^2(\sqrt{\omega}\dd\omega)}  \Bigg)^2 \Bigg]^{1/2}.
\end{align*} 
 Now we split again into cases: 

      \noindent
        \underline{Case 1: If $j_0 \geq j_1 - \text{Const.}$:} The resonant relation forces then that $j_0 \sim j_1$ (where we argue verbatim as in the cubic case). So there are only $\mathcal{O}(1)$ choices of $j_0$ for each $j_1$. 

Using the fixed dyadic estimate, we get
\begin{align*}
\|T_2(g,g)\|_{L^2(\sqrt{\omega}\dd\omega)}
&\lesssim
\sum_{j_1\ge j_2}
2^{j_1/4}
2^{-5j_1/4}
2^{j_2/4}
\|g_{j_1}\|_{L^2(\sqrt{\omega}\dd\omega)} \|g_{j_2}\|_{L^2(\sqrt{\omega}\dd\omega)} 
\\
&\qquad\qquad\times
\sum_{j_3\le j_1+\text{Const.}}
\min\left(1,2^{(\min(j_2,j_3)-j_1)/2}\right).
\end{align*}
Here we used $j_0\sim j_1$. We now estimate the
$j_3$-sum. Splitting into
$j_3\le j_2$ and $j_2<j_3\le j_1+\text{Const.}$, we obtain
\begin{align*}
\sum_{j_3\le j_1+\text{Const.}}
\min\left(1,2^{(\min(j_2,j_3)-j_1)/2}\right)
&\lesssim
\sum_{j_3\le j_2}2^{(j_3-j_1)/2}
+ \sum_{j_2<j_3\le j_1+\text{Const.}}2^{(j_2-j_1)/2}
\\
&\lesssim
(1+j_1-j_2)2^{-(j_1-j_2)/2}.
\end{align*}
Therefore
\begin{align*}
&\|T_2(g,g)\|_{L^2(\sqrt{\omega}\dd\omega)}
\lesssim
\sum_{j_1\ge j_2}
2^{j_1/4} 2^{-5j_1/4} 2^{j_2/4}
(1+j_1-j_2)2^{-(j_1-j_2)/2}\times \\
&\hskip 6.5cm \times \|g_{j_1}\|_{L^2(\sqrt{\omega}\dd\omega)} \|g_{j_2}\|_{L^2(\sqrt{\omega}\dd\omega)} 
\\
&\hskip 0.5cm=
\sum_{j_1\ge j_2}
(1+j_1-j_2) 2^{-j_1} 2^{j_2/4} 2^{-(j_1-j_2)/2}
\|g_{j_1}\|_{L^2(\sqrt{\omega}\dd\omega)} \|g_{j_2}\|_{L^2(\sqrt{\omega}\dd\omega)}.
\end{align*}
and 
\begin{align*}
\|T_2(g,g)\|_{L^2(\sqrt{\omega}\dd\omega)}
&\lesssim
\sum_{j_1\ge j_2}
(1+j_1-j_2)2^{-3(j_1-j_2)/4}
\|g_{j_1}\|_{L^2(\sqrt{\omega}\dd\omega)} \|g_{j_2}\|_{L^2(\sqrt{\omega}\dd\omega)}\\
&\lesssim
\|g\|_{L^2(\sqrt{\omega}\dd\omega)}^2, 
\end{align*}
since $(1+j_1-j_2)2^{-3(j_1-j_2)/4}$ is summable and by applying Cauchy-Schwarz and Young's inequality. 

\noindent
  \underline{Case 2: If $j_0 < j_1 - \text{Const.}$:} The resonant relation forces then that $j_3 \sim j_1$ (where we argue verbatim as in the cubic case). 
        Then \begin{align*}
           &\|T_2(g,g) \|_{L^2(\sqrt{\omega}\dd\omega)}\lesssim \\
& \sum_{j_1\geq j_2} 2^{-5j_1/4}2^{j_2/4}\|g_{j_1}\|_{L^2(\sqrt{\omega}\dd\omega)} \|g_{j_2}\|_{L^2(\sqrt{\omega}\dd\omega)} \left(\sum_{j_0<j_1 - \text{Const.}}  2^{j_0/2}\right)^{1/2} \lesssim  \\
& \sum_{j_1\geq j_2} 2^{-5j_1/4}2^{j_2/4}2^{j_1/4}\|g_{j_1}\|_{L^2(\sqrt{\omega}\dd\omega)} \|g_{j_2}\|_{L^2(\sqrt{\omega}\dd\omega)}=\\
& 
\sum_{j_1\geq j_2} 2^{-3j_1/4}2^{-(j_1-j_2)/4} \|g_{j_1}\|_{L^2(\sqrt{\omega}\dd\omega)} \|g_{j_2}\|_{L^2(\sqrt{\omega}\dd\omega)} 
\lesssim
\\ &  
\sum_{j_1\geq j_2} 2^{-(j_1-j_2)/4} \|g_{j_1}\|_{L^2(\sqrt{\omega}\dd\omega)} \|g_{j_2}\|_{L^2(\sqrt{\omega}\dd\omega)}  \lesssim \|g\|_{L^2(\sqrt{\omega})}^2
\end{align*}
where in the beginning we bounded $\sum_{j_0<j_1 - \text{Const.}}  2^{j_0/2}\lesssim 2^{j_1/2}$ and in the end that $2^{-(j_1-j_2)/4} \in \ell^1$, i.e. it is summable.

Again we have treated the higher-than-$1$ frequencies case, but the bounded case is immediate as explained in the cubic case. 

Finally let us comment on two things: first that we treated the case of one function $g$ in all the inputs, in order to keep the notation in the proof as simple as possible. To show it rather in the multi-linear case for different functions $g, h, \ell$ as stated in the proposition, it follows identically by placing the corresponding dyadic norms in the same estimates. 

And second, we first should prove the estimates for smooth compactly supported functions $g,h,\ell$. Since $C_c^\infty((0,\infty))$ is dense in $L^2(\sqrt{\omega}\dd\omega)$, the preceding bounds imply that $\Gamma_\mu$ and $Q_\mu$ extend uniquely to bounded multilinear maps
$ \Gamma_\mu: L^2(\sqrt{\omega}\dd\omega)^2
\to L^2(\sqrt{\omega}\dd\omega),
$
and
$Q_\mu: L^2(\sqrt{\omega}\dd\omega)^3
\to L^2(\sqrt{\omega}\dd\omega).$
The same estimates remain valid for these extensions. 
\end{proof}
We are thus ready to prove local well-posedness as follows. 

\begin{theorem}[Local Well-Posedness in the RJ perturbative regime]\label{Theo: LWP}
   Let $\mu>0$. The nonlinear equation \eqref{eq:NL evolution} is locally well-posed in $C^1([0,T];L^2(\sqrt{\omega}\dd\omega))$ for some $T= T(\|g_0\|_{L^2(\sqrt{\omega}\dd\omega)}, \mu)$.

Also, the solution extends uniquely to a maximal interval
$[0,T_{\max})$ for $T_{\max}\in(0,\infty]$, and satisfies the blow-up criterion
$$ T_{\max}<\infty \quad\Longrightarrow\quad
    \limsup_{t\uparrow T_{\max}}\|g(t)\|_{L^2(\sqrt{\omega}\dd\omega)}=\infty.$$
\end{theorem}
\begin{proof}
We remind that the equation reads
$$
\partial_t g=L_\mu g+\Gamma_\mu(g,g)+Q_\mu(g,g,g).
$$
and $L_\mu: L^2(\sqrt{\omega}) \to L^2(\sqrt{\omega})$ is a bounded operator.  
First the nonlinear map $\Gamma_\mu+Q_\mu$ is locally Lipschitz on $L^2(\sqrt{\omega})$: Indeed, let $g,h\in L^2(\sqrt{\omega} d\omega)$. Then we check directly that  
$$
\Gamma_\mu(g,g)-\Gamma_\mu(h,h)
=
\Gamma_\mu(g-h,g)+\Gamma_\mu(h,g-h), 
$$
and thus from Prop. \ref{prop: control NL}
\begin{align} 
\label{Lip1}
\|\Gamma_\mu(g,g)-\Gamma_\mu(h,h)\|_{L^2(\sqrt{\omega} \dd\omega ) }
\leq
C_\mu \Big(\|g\|_{L^2(\sqrt{\omega} \dd\omega)}+\|h\|_{L^2(\sqrt{\omega}\dd\omega)} \Big)
\|g-h\|_{L^2(\sqrt{\omega}\dd\omega)}.
\end{align}
Similarly, we check that 
\begin{align*} 
Q_\mu(g,g,g)-Q_\mu(h,h,h) = Q_\mu(g-h,g,g)
+ Q_\mu(h,g-h,g) + Q_\mu(h,h,g-h),
\end{align*}
and therefore from Prop. \ref{prop: control NL}, 
\begin{align} 
\label{Lip2}
\|Q_\mu(g,g,g)-Q_\mu(h,h,h)\|_{L^2(\sqrt{\omega})}
\leq
C_\mu
\left(
\|g\|_{L^2(\sqrt{\omega})}^2+\|g\|_{L^2(\sqrt{\omega})}\|h\|_{L^2(\sqrt{\omega})}+\|h\|_{L^2(\sqrt{\omega})}^2
\right)
\|g-h\|_{L^2(\sqrt{\omega})}.
\end{align}
So, if we define the ball
$$
\{g\in L^2(\sqrt{\omega}\dd\omega):\| g\|_{L^2(\sqrt{\omega} \dd\omega)} \leq R\},
$$
we have
\begin{align} \label{eq:NL_Lipschitz}
\|\mathcal (\Gamma_\mu + Q_\mu)(g)-(\Gamma_\mu +Q_\mu)(h)\|_{L^2(\sqrt{\omega}\dd\omega)}
\le
C_\mu (R+R^2)\|g-h\|_{L^2(\sqrt{\omega}\dd\omega)}.
\end{align}
Now fix $g_0\in L^2(\sqrt{\omega}\dd\omega)$, and for $T>0$, let $ X_T:=C([0,T];L^2(\sqrt{\omega}\dd \omega))$ with the norm $ 
\|g\|_{X_T}:=\sup_{0\le t\leq T}\|g(t)\|_{L^2(\sqrt{\omega} \dd\omega)}.$
Then we choose $ R = 2 \|g_0\|_{L^2(\sqrt{\omega}\dd\omega)} + 1$ and define the (closed) ball
$$
B_R:=\{g\in X_T:\|g\|_{X_T}\le R\}.
$$

Then the map 
$$ \Phi(g)(t) := g_0+\int_0^t
\left[ L_\mu g(s)+\Gamma_\mu(g(s),g(s))+Q_\mu(g(s),g(s),g(s))
\right] ds.
$$ satisfies for $T$ sufficiently small: is a contraction and  $\Phi(B_R)\subset B_R.$ 
Indeed for $g\in B_R$, the nonlinear estimates in Prop. \ref{prop: control NL}, and the boundedness
of $L_\mu$, give
$$  \|\Phi(g)\|_{X_T}\leq
    \|g_0\|_{L^2(\sqrt{\omega})} +
    T\Big[ \|L_\mu\|_{\mathscr{L}(L^2(\sqrt{\omega}))}R  +
        C_\mu(R^2+R^3)\Big].$$
Thus $\Phi(B_R)\subset B_R$, provided  
 $T\Big[\|L_\mu\|_{\mathscr{L}(L^2(\sqrt{\omega}))}R + C_\mu(R^2+R^3)\Big]
    \leq R-\|g_0\|_{L^2(\sqrt{\omega})}$. 
Similarly for the contraction: for $g,h\in B_R$, using
\eqref{eq:NL_Lipschitz}, we obtain
$$ \|\Phi(g)-\Phi(h)\|_{X_T} \leq
T \Big[ \|L_\mu\|_{\mathscr{L}(L^2(\sqrt{\omega}))}+ C_\mu(R+R^2)\Big]\|g-h\|_{X_T}.
$$
To conclude that $\Phi$ is a contraction on $B_R$, we choose $T>0$ so that 
\begin{equation}\label{eq:condition_contraction}
    T\Big[ \|L_\mu\|_{\mathscr L(L^2(\sqrt{\omega}))}+
        C_\mu(R+R^2)\Big]\leq\frac{1}{2},
\end{equation}

Thus existence of a unique solution in $C^1([0,T];L^2(\sqrt{\omega}))$ follows immediately from Banach fixed point theorem. Finally, by repeatedly applying the local result, the solution extends uniquely to a maximal interval $[0,T_{\max})$ from where we conclude the blow-up alternative.
\end{proof}

We finish this section by providing propagation of positivity for our local solutions. 

\begin{proposition}[Propagation of nonnegativity]
\label{prop:prop_positivity}
Let $\mu>0$, and $ g\in C^1([0,T];L^2(\sqrt{\omega} \dd\omega))
$ 
be a real-valued solution of \eqref{eq:NL evolution}, i.e. 
$$\partial_t g = \mathcal{F}_\mu(g) :=
 L_\mu g+ \Gamma_\mu(g,g) +Q_\mu(g,g,g)
$$
For solutions of the form $f(t,\omega)= n_\mu(\omega)(1+g(t,\omega))$, we have that if 
$ f(0,\omega)\geq0 $ for a.e.$\omega>0$, then  
$ f(t,\omega)\geq 0$ for a.e. $\omega>0$
for every $t\in[0,T_{max})$.
\end{proposition}

Let us show first the following lemma that we are going to use in the proof. 

\begin{lemma}
\label{lemm:positiv_F_mu}
Let $h_1,h_2 \in L^2(\sqrt{\omega}\dd\omega)$ be functions such that
$
h_1\geq-1$, $ h_2\geq0,$ and $ h_2(1+h_1)=0$ a.e. Then
$$
\big\langle
\mathcal{F}_\mu(h_1),h_2
\big\rangle_{L^2(\sqrt{\omega}\dd\omega)}
\geq 0.
$$
\end{lemma}
\begin{proof}
We have to be a bit careful since for RJ, each separate integral of the collisional operator may diverge. So in order to prove the lemma, we first consider this operator for truncated frequencies, $\mathcal{F}_{N,\mu}$, and then pass to the limit. In particular, we may consider a symmetric compact frequency cutoff $0 \leq X_N(\omega, \omega_1, \omega_2, \omega_3)\leq 1$ which is supported whenever $\omega_i\in [N^{-1}, N]$ for all $i=0,1,2,3$: e.g.
Let $\chi_N\in C_c^\infty((0,\infty))$ satisfy
$ 0 \leq\chi_N \leq 1,$
and $ \operatorname{supp}\chi_N \subset[N^{-1},N].$ 
We define the symmetric frequency cutoff
$$
X_N(\omega_0,\omega_1,\omega_2,\omega_3)
:= \prod_{i=0}^3 \chi_N(\omega_i), $$
and the truncated nonlinear operator 
\begin{align*}
\mathcal{F}_{\mu,N}(h_1)(\omega_0)
:=-n_\mu^{-1} \iint_{D(\omega_0)}
X_N \Omega(\omega_0,\omega_1,\omega_2,\omega_3)
q \big( n_\mu(1+h_1) \big)
\dd\omega_1\dd\omega_2.
\end{align*}
For fixed $N$, all frequencies are restricted to a compact subset of
$(0,\infty)$, and this collision integral is well defined. 
Set
$$
f_i^{h_1}:=n_\mu(\omega_i) (1+h_1(\omega_i)) \geq 0.
$$
On the set where $h_2>0$, the relation $h_2(1+h_1)=0$ implies
$1+h_1=0,$ and so $ f_0^{h_1}=0$. 
Therefore, using the definition of $q$,
\begin{align*}
q(f^{h_1})
&=
f_0^{h_1} f_2^{h_1}f_3^{h_1} +
f_0^{h_1}f_1^{h_1}f_3^{h_1}-
f_1^{h_1}f_2^{h_1}f_0^{h_1}-
f_1^{h_1}f_2^{h_1}f_3^{h_1}
=-f_1^{h_1}f_2^{h_1}f_3^{h_1}.
\end{align*}
Consequently, on the set $\{h_2>0\}$,
$$
\mathcal F_{\mu,N}(h_1)(\omega_0)
= n_\mu^{-1} \iint_{D(\omega_0)}
X_N\Omega
f_1^{h_1}f_2^{h_1}f_3^{h_1}
\dd\omega_1\dd\omega_2 \geq 0.
$$
Since $h_2 \geq 0$, it follows that
\begin{equation}
\label{eq:truncated_quasi_positivity}
\big\langle
\mathcal{F}_{\mu,N}(h_1),h_2
\big\rangle_{L^2(\sqrt{\omega}\dd\omega)}
\geq 0.
\end{equation} 

We now pass to the limit $N\to\infty$. First, if $h\in C_c^\infty((0,\infty))$,  since $X_N \to 1$ pointwise and $0 \leq X_N\leq 1$, dominated convergence
(which we can apply since the proofs of the corresponding operator estimates gives us the integrable majorants) gives 
$$
\mathcal{F}_{\mu,N}(h) \to
\mathcal{F}_\mu(h)
\qquad\text{in }L^2(\sqrt{\omega}\dd\omega).
$$

Now for a general $h\in L^2(\sqrt{\omega}\dd\omega)$: by density, choose
$h^{(m)}\in C_c^\infty((0,\infty))$ such that
$
h^{(m)}\to h $ in $L^2(\sqrt{\omega}\dd\omega)$.

Moreover, the proof of Proposition~\ref{prop: control NL} applies
verbatim to the truncated operators. Indeed, the same algebraic
cancellations are preserved and, after taking absolute values, the
common factor $X_N$ can be discarded since $0\leq X_N\leq1$.
Consequently, on every ball of radius $R$,
$$
\|\mathcal{F}_{\mu,N}(h)-\mathcal{F}_{\mu,N}(k)\|_{L^2(\sqrt{\omega}\dd\omega)}
\leq
C_{\mu,R}
\|h-k\|_{L^2(\sqrt{\omega}\dd\omega)},
$$
with a constant independent of $N$.

Choosing $R$ so that $h$ and the sequence $h^{(m)}$ belong to the ball
of radius $R$, we have
\begin{align*}
&\|\mathcal{F}_{\mu,N}(h)-\mathcal{F}_\mu(h)\|_{L^2(\sqrt{\omega}\dd\omega)}
\leq
\|\mathcal{F}_{\mu,N}(h)-\mathcal{F}_{\mu,N}(h^{(m)})\|_{L^2(\sqrt{\omega}\dd\omega)}
\\ & \hspace{0.2cm} + \|\mathcal{F}_{\mu,N}(h^{(m)})-\mathcal{F}_\mu(h^{(m)})\|_{L^2(\sqrt{\omega}\dd\omega)} +
\|\mathcal{F}_\mu(h^{(m)})-\mathcal{F}_\mu(h)\|_{L^2(\sqrt{\omega}\dd\omega)}
\\ & \hspace{0.4cm}\leq
2C_{\mu,R}\|h -h^{(m)}\|_{L^2(\sqrt{\omega}\dd\omega)} +
\|\mathcal{F}_{\mu,N}(h^{(m)})-\mathcal{F}_\mu(h^{(m)})\|_{L^2(\sqrt{\omega}\dd\omega)}.
\end{align*}
First choosing $m$ sufficiently large and then letting $N\to\infty$ proves the limit.
Applying this for $h=h_1$, we may pass to the limit
in \eqref{eq:truncated_quasi_positivity} and conclude:
$$
\big\langle \mathcal{F}_\mu(h_1), h_2
\big\rangle_{L^2(\sqrt{\omega} \dd\omega)}
= \lim_{N \to \infty}
\big\langle \mathcal{F}_{\mu,N}(h_1),h_2
\big\rangle_{L^2(\sqrt{\omega} \dd\omega)}
\geq 0.
$$
\end{proof}

\begin{proof}[Proof of Proposition \ref{prop:prop_positivity}]
We look at the relative density $1+g = f/n_\mu$ and we decompose $1+g = (1+g)^+ - (1+g)^-$. We define $\overline{g}=\max\{g,-1\} = g + (1+g)^- = (1+g)^+-1$ and  notice that both $(1+g)^-$ and $\overline{g} \in L^2(\sqrt{\omega}\dd \omega)$ since $0\leq (1+g)^- \leq |g|$ and $0\leq |\overline{g}| \leq |g|$, and so  $\|(1+g)^-\|_{L^2(\sqrt{\omega}\dd \omega  )}, \|\overline{g}\|_{L^2(\sqrt{\omega}\dd \omega)} \leq \|g\|_{L^2(\sqrt{\omega}\dd \omega)}$.

We define the functional $$\mathcal{E}(t):= \frac{1}{2} \int_0^\infty | (1+g_t)^-|^2(\omega)\sqrt{\omega}\dd\omega,$$
and then we time-differentiate 
\begin{align*}
   \mathcal{E}'(t)& = -\Big\langle  \mathcal{F}_\mu(g_t), (1+g_t)^-\Big\rangle_{L^2(\sqrt{\omega}\dd\omega)}  \\
   & = -\Big\langle \left[ \mathcal{F}_\mu(g_t) - \mathcal{F}_\mu(\overline{g}_t)\right] , (1+g_t)^- \Big\rangle_{L^2(\sqrt{\omega}\dd \omega)} - 
   \Big\langle \mathcal{F}_\mu(\overline{g}_t) , (1+g_t)^-\Big\rangle_{L^2(\sqrt{\omega}\dd \omega)},
\end{align*}
where in the second line we added and subtracted the $\mathcal{F}_\mu(\overline{g}_t)$-term. Now we claim that 
\begin{align} \label{eq:claim1_positiv}
\big\langle \mathcal{F}_\mu(\overline{g}_t) , (1+g_t)^-\big\rangle_{L^2(\sqrt{\omega}\dd \omega)}\geq 0.
\end{align}
Indeed we see this by an application of Lemma \ref{lemm:positiv_F_mu}, with
$h_1=\overline{g}_t$, $h_2 = (1+g_t)^- $.
 
 We fix $T <T_{max}$ and set $ R:=\sup_{0\leq t\leq T }\|g(t)\|_{L^2(\sqrt{\omega}\dd \omega)}$. By (\ref{Lip1}) and (\ref{Lip2}),  $\mathcal{F}_\mu$ is locally Lipschitz on $L^2(\sqrt{\omega}\dd \omega )$ and
\begin{align*}
     \|\mathcal{F}_\mu(g_t)-\mathcal{F}_\mu(\tilde{g}_t)\|_{L^2(\sqrt{\omega}d\omega )}\leq
    C_{\mu,R}\|g-\tilde{g}_t\|_{L^2(\sqrt{\omega}\dd \omega )}
=C_{\mu,R}\|(1+g_t)^-\|_{L^2(\sqrt{\omega}\dd \omega )}, 
\end{align*}
with $C_{\mu,R} = \|L_\mu\|_{\mathcal{L}(L^2(\sqrt{\omega}d\omega ))} + C_\mu (R+R^2)$. 
Thus combining it with Cauchy-Schwarz, we get 
\begin{align*}
    \mathcal{E}'(t)& \leq C_{\mu, R}\|(1+g_t)^-\|_{L^2(\sqrt{\omega}\dd \omega)}^2 =  2 C_{\mu,R} \mathcal{E}(t).  
\end{align*}
   Gr\"{o}nwall's inequality then yields $\mathcal{E}(t)\leq e^{2 C_{\mu,R} t} \mathcal{E}(0)$. Assuming that $f_0\geq 0$ means that $g_0 \geq -1$ and so $(1+g_0)^-=0$ or that $\mathcal{E}(0)=0$. Therefore $\mathcal{E}(t)=0$ and so $1+g(t,\omega)\geq 0$ on $[0,T_{max}]$ and for a.e. $\omega$, from where we conclude the Proposition.

%   Now it remains to verify \eqref{eq:claim1_positiv}: To see this we denote by $\overline{f}_i = n_\mu(1+\overline{g}_i)$, which is $\geq 0$ (since $1+\overline{g} \geq 0$). Also observe that having that $(1+g)(\omega_0)^->0$ means that $g(\omega_0)<-1$ and thus $\overline{g}(\omega_0)=-1 \Rightarrow \overline{f}(\omega_0)=0$. We then write 
%   \begin{align*}
 %  \big\langle \mathcal{F}_\mu(\overline{g}_t) , (1+g)^-\big\rangle_{L^2(\sqrt{\omega})}
%     &=\big\langle n_\mu^{-1} \mathcal{C}\big(n_\mu (1+\overline{g}_t) \big), (1+g)^- \big\rangle_{L^2(\sqrt{\omega})}
 %    \\& = \left\langle -n_\mu^{-1} \iint_{D(\omega)} \Omega (\omega , \omega _1, \omega _2, \omega _3) q(\overline{f}) \dd\omega_1\dd\omega_2 , (1+g)^- \right\rangle_{L^2(\sqrt{\omega})}
  %     \\&= \left\langle n_\mu^{-1}
 %   \iint_{D(\omega)}
  %  \Omega (\omega, \omega_1, \omega_2, \omega_3)
%    \overline{f}_1\overline{f}_2\overline{f}_3
  %  \dd\omega_1\dd\omega_2,(1+g)^- \right\rangle_{L^2(\sqrt{\omega})}  \ \geq 0, 
%   \end{align*} 
 %  since the three terms containing $\overline{f}_0$ in the collisional integral are zero, due to the observation above.    
\end{proof}

\section{Spectral Gap for the Linear Operator} \label{sec:SG}

In Section \ref{subsec:Boundedness_NoComp_KerL}, we showed that the $K_\mu$ part of the linear operator is not compact, and thus the standard splitting of "multiplication operator + compact" is not the proper tool, in this setting, to conclude a spectral gap for the linear operator $L_\mu$. Nevertheless, for $\mu>0$, in the following subsection we show that there exists a spectral gap, and so the essential spectrum of $L_\mu$ does not touch $0$.  

\subsection{Spectral gap for the linear operator when $\mu>0$}

In particular we obtain the following Theorem: 
 
\begin{theorem}[Poincaré Inequality/Spectral Gap Inequality] \label{theo:SGI}
    There exists $\lambda_\mu >0$ so that 
    $$ D_{\mu}(g) = - \langle L_\mu g,g \rangle \geq \lambda_\mu \| (I - \Pi)g\|_{L^2(\sqrt{\omega}\dd \omega)}^2, \quad \text{ for all } g \in L^2(\sqrt{\omega}\dd\omega). $$
\end{theorem}
From \eqref{eq: A_mu+K_mu}, the linear operator is $$L_\mu g  = \iint_{D(\omega)} \Omega (\omega , \omega _1, \omega _2, \omega _3) n_1n_2n_3 \Big[-\frac{g_0}{n}  - \frac{g_3}{n_3} + \frac{g_1}{n_1} + \frac{g_2}{n_2}   
\Big] \dd \omega_1 \dd \omega_2 =: -[A_\mu g](\omega) + [K_\mu g](\omega).$$
We start with some useful Lemmas and Propositions that we are going to use in the proof of the Poincaré Inequality. 

\begin{lemma}\label{lemm:A_mu}
   The operator $A_\mu$ in \eqref{eq: A_mu+K_mu}, is a multiplication operator with $[A_\mu g](\omega) = a_\mu(\omega) g(\omega)$ and the function $a_\mu$ is lower bounded away from zero. In particular  
   $$ a_\mu(\omega) \geq \frac{\pi^2}{6} \text{ for all } \omega>0.$$
\end{lemma}

\begin{proof}
The function $a_\mu(\omega ) $ is
\begin{align*}
 a_\mu (\omega )= -\iint_{ D(\omega ) }\frac {\min\{\sqrt{\omega }, \sqrt{\omega _1}, \sqrt{\omega _2}, \sqrt{\omega _3}\}} {\sqrt{\omega }}\left(n_{\mu,1} n_{\mu,2} - n_{\mu,3}n_{\mu,1}-n_{\mu,3}n_{\mu,2} \right)\dd \omega _1 \dd \omega_2
   \end{align*}
   with,
   \begin{align*}
&n_{\mu,1}n_{\mu,2}-n_{\mu,3}n_{\mu,1}-n_{\mu,3}n_{\mu,2}=\frac{1}{(\mu +\omega _1)}\frac {1} {( \mu +\omega _2)}-\frac {1} {(\mu +\omega _1+\omega _2-\omega )}\frac {1} {(\mu +\omega _1)} \\
&\hskip 4cm -\frac {1} {(\mu +\omega _1+\omega _2-\omega )}\frac {1} {(\mu +\omega _2)}\\
&= -\frac {\mu +\omega } {(\mu +\omega _1+\omega _2-\omega )(\mu +\omega _1)(\mu +\omega _2)}.
\end{align*}
Since, by definition,  $\omega _1+\omega _2\ge \omega $ on $D(\omega )$, it follows
\begin{align*}
a_\mu (\omega )&\ge (\mu +\omega)
\int\int_{\omega _1>\omega , \omega _2>\omega }\frac {\dd \omega _1 \dd \omega _2 } {(\mu +\omega _1+\omega _2-\omega )(\mu +\omega _1)(\mu +\omega _2)}\\
&=(\mu +\omega )\int 
 _{ \omega _1>\omega  }    
 \log\left(\frac{\mu +\omega_1 }{\mu +\omega } \right)\frac{\dd \omega_1 }{(\mu +\omega _1)(\omega -\omega _1)}
 =\frac{\pi^2}{6}.
 \end{align*}

\end{proof}

Then we state and prove two Propositions regarding coercivity in low frequencies but also separately in high frequencies.

\underline{We start with the low frequencies:} 

\begin{proposition}[Low frequency estimate] \label{lemm:Tightness_zero}
    Let $\mu>0$. There exists $C_\mu>0$ so that for $0<M_1\leq \mu/2$, 
    \begin{align} \label{eq:Tightness_zero_1}
    \|\chi_{(0,M_1)}K_\mu g\|_{L^2(\sqrt{\omega})} \lesssim_\mu M_1^{3/4} \|g\|_{L^2(\sqrt{\omega})}. 
    \end{align}
    Moreover, 
    \begin{align} \label{eq:Tightness_zero_2}
        \|\chi_{(0,M_1)} g\|_{L^2(\sqrt{\omega})}^2 \lesssim_\mu D_\mu(g) +  M_1^{3/2} \|g\|_{L^2(\sqrt{\omega})}^2.
    \end{align}
    In particular, if 
    $$ (g_n)_n \subset L^2(\sqrt{\omega})\ \text{ with } \ \|g_n\|_{L^2(\sqrt{\omega})}=1 \ \text{ and }
    D_\mu(g_n)\to 0,$$
    then  
\begin{align}\label{eq:Tightness_zero_3}
\lim_{M_1\downarrow0}\limsup_{n\to\infty}
    \|\chi_{(0,M_1)}g_n\|_{L^2(\sqrt{\omega})}=0.
    \end{align}
\end{proposition}
\begin{proof}
    Let us consider the $g_1$-term in $K_\mu$. The $g_2$-term is identical. We write:  
    $$I_1(\omega):= \iint_{D(\omega)} \Omega\  n_2n_3g_1\,d\omega_1d\omega_2.$$
    Then 
    \begin{align} \label{eq:tightat0_1}
        |I_1|\leq \int_0^\infty |g(\omega_1)|\left[ \int_{(\omega-\omega_1)_+}^\infty \frac{\dd \omega_2}{(\omega_2+\mu)(\omega_1+\omega_2-\omega + \mu)} \right] \dd \omega_1.
    \end{align}
    We estimate the inner $\omega_2$-integral: For $0<\omega\le\mu/2$, we get
\begin{align}\label{eq:tightat0_2}\int_{(\omega-\omega_1)_+}^\infty \frac{\dd \omega_2}{(\omega_2+\mu)(\omega_1+\omega_2-\omega + \mu)}
    \le C_\mu
    \frac{1+\log(1+\omega_1/\mu)}{\omega_1+\mu}.
\end{align}
By Cauchy-Schwarz in \eqref{eq:tightat0_1}, 
\begin{equation}\label{eq:tightat0_3}
\begin{split}
     |I_1| 
    &\leq \|g\|_{L^2(\sqrt{\omega}\dd\omega)} \left[ \int_0^\infty \left( \int_{(\omega-\omega_1)_+}^\infty \frac{\dd \omega_2}{(\omega_2+\mu)(\omega_1+\omega_2-\omega + \mu)}\right)^2 \frac{\dd\omega_1}{\sqrt{\omega_1}}\right]^{1/2}
    \\
    & \leq C_\mu \|g\|_{L^2(\sqrt{\omega}\dd\omega)}, 
\end{split}
\end{equation}
since given the computation in \eqref{eq:tightat0_2}, the above integral is finite uniformly in $0<\omega\le\mu/2$: Indeed, near $0$, the integrand behaves like $\omega_1^{-1/2}$, and near infinity it behaves like $(\log \omega_1)^2\omega_1^{-5/2}$. 

For the $g_3$-term, we argue similarly. First for fixed $\omega_3$, we obtain
$$\int_0^{\omega+\omega_3}
    \frac{d\omega_1}
    {(\omega_1+\mu)(\omega+\omega_3-\omega_1+\mu)}
    \le C_\mu
    \frac{1+\log(1+\omega_3/\mu)}{\omega_3+\mu}, 
$$ 
and finally again by Cauchy-Schwarz we conclude that 
$$ |K_\mu g(\omega)|
    \le C_\mu\|g\|_{L^2(\sqrt{\omega}\dd\omega)},
    \qquad 0<\omega\le\mu/2.
$$
Therefore, 
\begin{equation}
    \|\chi_{(0,M_1)}K_\mu g\|_{L^2(\sqrt{\omega}\dd\omega)}^2
    \le
    C_\mu\|g\|_H^2
    \int_0^{M_1}\sqrt{\omega} \dd\omega \le
    C_\mu M_1^{3/2}\|g\|_{L^2(\sqrt{\omega}\dd\omega)}^2, 
\end{equation}
which concludes the first claim, \eqref{eq:Tightness_zero_1} of the Proposition.

For the rest, we use the decomposition $ A_\mu g=-L_\mu g+K_\mu g,
$ and the lower bound $a_\mu\ge \frac{\pi^2}{6}$ from Lemma \ref{lemm:A_mu}. Then  
\begin{equation}
\begin{split}
    \frac{\pi^2}{6}\|\chi_{(0,M_1)}g\|_{L^2(\sqrt{\omega}\dd\omega)}
    &\le \|\chi_{(0,M_1)}A_\mu g\|_{L^2(\sqrt{\omega}\dd\omega)} \\
    &\le
    \|-L_\mu g\|_{L^2(\sqrt{\omega}\dd\omega)}
    +\|\chi_{(0,M_1)}K_\mu g\|_{L^2(\sqrt{\omega}\dd\omega)}.
\end{split}
\end{equation}

Now since $-L_\mu$ is bounded, self-adjoint and nonnegative, it admits a bounded nonnegative square root, and so
$$ \|L_\mu g\|_{L^2(\sqrt{\omega}\dd\omega)}^2 =\|\sqrt{-L_\mu} \sqrt{-L_\mu}g\|_{L^2(\sqrt{\omega}\dd\omega)}^2 \le \|\sqrt{-L_\mu}\|^2\|\sqrt{-L_\mu}g\|_{L^2(\sqrt{\omega}\dd\omega)}^2 = \|L_\mu\| D_\mu(g).
$$
%Now since $-L_\mu$ is bounded, self-adjoint and nonnegative,
%$$ \|-L_\mu g\|_{L^2(\sqrt{\omega}\dd\omega)}^2
 %   \le \|L_\mu\| \langle (-L_\mu g),g \rangle =
  %  \|L_\mu\| D_\mu(g).$$ 

Thus
    $$
\|\chi_{(0,M_1)}g\|_{L^2(\sqrt{\omega}\dd\omega)}
\le C_\mu D_\mu(g)^{1/2}
    + C_\mu M_1^{3/4}\|g\|_{L^2(\sqrt{\omega}\dd\omega)}.
$$
This concludes the second claim, \eqref{eq:Tightness_zero_2}, of the Proposition. 
The last claim \eqref{eq:Tightness_zero_3} is immediate after taking the $\limsup_n$. 
\end{proof}

\underline{We continue with the high frequencies} which are of course  more delicate. Indeed due to the lack of compactness of $K_\mu$, we cannot expect a similar estimate \eqref{eq:tightat0_1} as in the low-frequencies case in Lemma \ref{lemm:Tightness_zero}!

However we are still able to recover an analogous 'coercivity property at infinity' 
$$ \|\chi_{(M_2, \infty)}g\|_{L^2(\sqrt{\omega})}^2 \lesssim D_\mu (g) + \log(M_2)^{-2} \|g\|_{L^2(\sqrt{\omega})}^2, $$
by (i) keeping only a positive contribution of the Dirichlet form where all four frequencies $\omega_i$ are all comparable: $$D_\mu(g)\geq D_\mu^{restr}(g)$$
where $D_\mu^{restr}$ is the restricted Dirichlet form in a bounded rectangle $\mathcal{R} = [R_1,R_2]^2$, with $R_1\geq 1$ where all four frequencies are comparable,\\ (ii) provide a coercivity property for $D_{\mu=0}^{restr}$ when $\mu=0$ (which deals with functions $g$ supported at infinity) cf Proposition \ref{Prop:SG_at_infty},  and finally \\ (iii) transfer a coercivity property to arbitrary functions $g$ through a localisation argument with a smooth cut-off function, cf Proposition \ref{lemm:transfer_coercivity}.

We set $$\omega_1=\alpha \omega, \omega_2=\beta \omega \text{ and } \gamma:=\alpha+\beta-1 \text{ so that } \omega_3=\gamma \omega.$$  They all live in a domain  $\mathcal{R} \subset \{(\alpha,\beta) \in (0,\infty)^2: \alpha+\beta>1\}$. We assume that $\mathcal{R} = [R_1,R_2]^2$ with $R_1\geq 1$ and $R_2$ bounded. 
Then for $\mu\geq 0$, the Dirichlet form is 

\begin{equation}
    \begin{split}
    &D_\mu(g) \geq  D_\mu^{restr}(g) := \\
    & \frac{1}{4} \int_0^\infty \iint_{\mathcal{R}} \frac{\sqrt{\omega} \min(1,\sqrt{\alpha}, \sqrt{\beta},\sqrt{\gamma}) }{(\omega+\mu)(\alpha\omega+\mu)(\beta\omega+\mu)(\gamma\omega+\mu)}\left[- h_\mu(\alpha \omega) - h_\mu(\beta \omega) + h_\mu(\gamma \omega) + h_\mu(\omega)\right]^2 \omega^2\dd \alpha \dd \beta  \dd \omega 
    \\
    & = \frac{1}{4} \int_0^\infty \iint_{\mathcal{R}} \frac{\omega^{5/2}}{(\omega+\mu)(\alpha\omega+\mu)(\beta\omega+\mu)(\gamma\omega+\mu)}\left[- h_\mu(\alpha \omega) - h_\mu(\beta \omega) + h_\mu(\gamma \omega) + h_\mu(\omega)\right]^2 \dd \alpha \dd \beta  \dd \omega
    \end{split}
\end{equation}
where $h_\mu(z):=(z+\mu) g(z)$, and where we used that 
 $\dd\omega_1\dd\omega_2 = \omega^2 \dd \alpha \dd \beta$ and we can fix $\mathcal{R}$ so that the minimum factor is $\sqrt{\omega}$.

Also, at high frequencies, $\mu$ is negligible and thus the relevant object is the above restricted Dirichlet form at $\mu=0$, namely: 
$$ D_0^{restr}(h):= 
\int_0^\infty  \omega^{-3/2} \iint_{\mathcal{R}}  \frac{1}{4\alpha\beta\gamma} \left[- h(\alpha \omega) - h(\beta \omega) + h(\gamma \omega) + h(\omega) \right]^2  \dd \alpha \dd \beta \dd \omega, $$
where $$h(z):=zg(z).$$

\begin{proposition}[Coercivity for the limiting Dirichlet form \& Spectral Gap for functions supported at infinity] \label{Prop:SG_at_infty}
There exists $\lambda_{\mathcal{R}}>0$ such that for all $h=\omega g \in L^2(\omega^{-3/2}\dd\omega)$, we have 
\begin{equation}\label{eq:SG_at_infty_1}
    D_0^{restr}(h) \ge
    \lambda_{\mathcal{R}}
    \int_0^\infty |h(\omega)|^2\omega^{-3/2} \dd \omega.
\end{equation} 
As a consequence, we get coercivity for functions supported at infinity: There exists $M^*>0$ so that whenever $\operatorname{supp} g \subset [M^*, \infty)$,
\begin{equation}\label{eq:SG_at_infty_2}
   D_\mu (g) \geq D_\mu^{restr}(g) \geq \lambda_{\infty} \|g\|_{L^2(\sqrt{\omega}\dd\omega)}^2
\end{equation} 
for some $\lambda_{\infty}>0$. 
\end{proposition}
\begin{proof}
    We start with proving \eqref{eq:SG_at_infty_1}. We consider the problem in the new Fourier variables $\omega =e^x $ and define $G(x) := h(e^x) e^{-x/4}$ so that 
    \begin{align} \label{eq:SG_at_inftyProof1}
        \int_0^\infty |h(\omega)|^2\omega^{-3/2} \dd\omega = \int_{-\infty}^\infty |G(x)|^2 \dd x.
    \end{align} 
    In these variables, the collisional integrand is
\begin{align*}
    &- h(\alpha \omega) - h(\beta \omega) + h(\gamma \omega) + h(\omega) = - h(e^{\log \alpha + x}) - h(e^{\log \beta + x}) + h(e^{\log \gamma + x}) + h(e^x)
    \\ & = e^{x/4} \Big( -\alpha^{1/4} G(x+\log \alpha) -\beta^{1/4} G(x+\log \beta) + \gamma^{1/4} G(x+\log \gamma) + G(x) \Big)  
\end{align*}
while for the Jacobian we have $\omega^{-3/2}\dd\omega = e^{-x/2}\dd x$. Then by Plancherel, and since 
$$ \mathcal{F}(G(\cdot+\log \alpha))(\xi) = \alpha^{i\xi} \mathcal{F}(G)(\xi)$$ we have that 
$$
D_0^{restr}(h) = \int_{-\infty}^{+\infty} |\widehat{G}(\xi)|^2 \iint_{\mathcal{R}} \frac{\big\vert -\alpha^{1/4+i\xi} - \beta^{1/4+i\xi}+
            \gamma^{1/4+i\xi}+1
        \big\vert^2}{4\alpha\beta\gamma} 
 \dd\alpha\dd\beta\ \dd \xi.
$$
To complete the proof, it remains to show that \begin{align}\label{eq:SG_at_inftyProof2}
    \lambda_{\mathcal{R}}:= \inf_{\xi \in \mathbb{R}} \iint_{\mathcal{R}} \frac{\big\vert -\alpha^{1/4+i\xi} - \beta^{1/4+i\xi}+
            \gamma^{1/4+i\xi}+1
        \big\vert^2}{4\alpha\beta\gamma} 
 \dd\alpha\dd\beta > 0.
\end{align}
Indeed if we know that then combining it with \eqref{eq:SG_at_inftyProof1}, we immediately obtain \eqref{eq:SG_at_infty_1}. 

In order to see that \eqref{eq:SG_at_inftyProof2} holds: 
First we observe that we have a pointwise positivity: for every fixed $\xi \in \mathbb{R}$, the double integral is positive. Indeed, if not then $F_\xi (\alpha, \beta):= \alpha^{1/4+i\xi} + \beta^{1/4+i\xi}-
            (\alpha+\beta-1)^{1/4+i\xi}-1=0$ for a.e. pair $(\alpha, \beta) \in \mathcal{R}$. By continuity of $F_\xi$, in fact the identity should hold for all pairs $(\alpha, \beta) \in \mathcal{R}$. And if we differentiate w.r.t. $\alpha$ for example, we have $(1/4+i\xi)\alpha^{i\xi -3/4} = (\alpha+\beta-1)^{-3/4+i\xi}(1/4+i\xi)$, which is true only if $i\xi=3/4$, but $\xi \in \mathbb{R}$, so it is impossible. 

Now to pass to a rather uniform lower bound we need to show that \begin{align}\label{eq:SG_at_infty_4} \iint_{\mathcal{R}} \frac{\big\vert -\alpha^{1/4+i\xi} - \beta^{1/4+i\xi}+
            \gamma^{1/4+i\xi}+1
        \big\vert^2}{4\alpha\beta\gamma} 
 \dd\alpha\dd\beta > 0 \quad \text{ as } |\xi| \to \infty,\end{align}
 since as a function of $\xi$ this is continuous and pointwise positive and thus if \eqref{eq:SG_at_infty_4} is also true then $\lambda_{\mathcal{R}}>0$.
 
 For \eqref{eq:SG_at_infty_4}: If we expand the square, we get two kinds of terms: four of the form $\ell^{1/4+i\xi}\overline{\ell^{1/4+i\xi}} = \ell^{1/2}$, for $\ell\in\{\alpha,\beta, \gamma, 1\}$,  and six mixed term that will vanish by non-stationary phase. Indeed these mixed terms are of the form
 either $ \pm 2 \ell^{1/4+i\xi} \overline{\kappa^{1/4+i\xi}} = 2 (\ell\kappa)^{1/4}e^{i \xi \log (\ell/\kappa)}$ or $\pm \ell^{1/4}e^{i \xi \log (\ell)}$ for $\ell, \kappa \in\{\alpha,\beta, \gamma\}$. So these cross terms lead to oscillatory integrals of the form $$\iint_{\mathcal{R}} \psi(\alpha, \beta)e^{i\xi \phi(\alpha, \beta)}\dd\alpha \dd \beta $$
 with each of the phases $\phi$ \footnote{The possible phases are
$\log\alpha$, $\log\beta$,  $\log\gamma$, $\log(\alpha/\beta)$, $\log(\alpha/\gamma)$,
$\log(\beta/\gamma)$.} having non vanishing gradient on the rectangle  $\mathcal{R}$. And since for example $e^{i\xi \phi(\alpha, \beta)} = \partial_\alpha (e^{i\xi \phi(\alpha, \beta)}) \frac{1}{i \xi\partial_\alpha \phi }$, by integration by parts (each time integrate by parts in the appropriate variable), each of these terms will behave as $\mathcal{O}(|\xi|^{-1})$. Thus the mixed terms vanish as $|\xi| \to \infty$. In other words, 
 $$ \iint_{\mathcal{R}} \frac{\big\vert -\alpha^{1/4+i\xi} - \beta^{1/4+i\xi}+
            \gamma^{1/4+i\xi}+1
        \big\vert^2}{4\alpha\beta\gamma} 
 \dd\alpha\dd\beta \to \iint_{\mathcal{R}} \frac{\alpha^{1/2}+\beta^{1/2}+\gamma^{1/2}+1}{4\alpha\beta\gamma} 
 \dd\alpha\dd\beta > 0$$
 as $|\xi| \to \infty$, and
 so the claim follows. 
 
 To see the second claim \eqref{eq:SG_at_infty_2}, we will compare the integral kernel of $D_\mu^{restr}$ at $\mu>0$ with the one at $\mu=0$, $D_0^{restr}$, and will apply the first part of the Lemma.
 Indeed $$\frac{\omega^{5/2}}{(\omega+\mu)(\alpha\omega+\mu)(\beta\omega+\mu)(\gamma\omega+\mu)} = \frac{\omega^{-3/2}}{\alpha\beta\gamma}
    \prod_{\ell \in\{1,\alpha,\beta,\gamma\}}
    \frac{\ell \omega}{\ell\omega+\mu}.$$
    We are going to lower bound that product on the RHS using that all the frequencies are comparable on $\mathcal{R} =[R_1,R_2]^2$. 
    
  For functions supported in $[M^*,\infty)$ with the integrand of $D_0^{restr}(h_\mu)$ nonzero, means that at least one of $\omega,\alpha\omega, \beta\omega, \gamma\omega $ are in $[M^*,\infty)$. Also since we are restricted on the bounded rectangle $\mathcal{R} =[R_1,R_2]^2$ with $1\leq R_1\leq \alpha, \beta \leq R_2$, $2R_1-1\leq \gamma\leq 2R_2-1$ and so for $M^* \geq C(\mathcal{R}) \mu$, with $C(\mathcal{R}):=\sup_{(\alpha, \beta) \in \mathcal{R} } 
  \max\{1,\alpha,\beta,\gamma\}<\infty$, we have:  
  $$ \ell\omega\geq \omega \geq \frac{M^*}{C(\mathcal{R})} \geq \mu, \quad \text{ for } \ell \in\{1,\alpha,\beta,\gamma\}.$$ This yields 
  $$\frac{\ell \omega}{\ell\omega+\mu} \geq \frac{1}{2} \ \Longrightarrow\  \ \frac{\omega^{-3/2}}{\alpha\beta\gamma}
    \prod_{\ell \in\{1,\alpha,\beta,\gamma\}}
    \frac{\ell \omega}{\ell\omega+\mu}\geq \frac{1}{16} \frac{\omega^{-3/2}}{\alpha\beta\gamma}. $$
    Thus
 \begin{equation}
     \begin{split}
         D_\mu (g) \geq D_\mu^{restr}(g)&\geq \frac{1}{16}D_0^{restr}(h_\mu)
  \geq \frac{\lambda_{\mathcal{R}}}{16} \int_0^\infty |h_\mu (\omega)|^2\omega^{-3/2} \dd \omega \\ & =  \frac{\lambda_{\mathcal{R}}}{16}\int_0^\infty
    |g(\omega)|^2\sqrt{\omega}
    \left(1+\frac{\mu}{\omega}\right)^2
    \dd \omega \geq \frac{\lambda_{\mathcal{R}}}{16} \|g\|_{L^2(\sqrt{\omega})}^2
     \end{split}
 \end{equation}
 where in the second inequality we applied the lower bound on the integral kernel computed just above, in the third inequality we applied the first part of the Proposition, \eqref{eq:SG_at_infty_1}, and then that $h_\mu(z)=(z+\mu)g(z)$. 
\end{proof}

\begin{lemma}[Transfer coercivity to arbitrary functions] \label{lemm:transfer_coercivity}
    There exists $M^*>0$ so that for all $M>M^*$
    \begin{equation} \label{eq:lemm:transfer_coercivity1}
        \begin{split}
\|\chi_{(M,\infty)}g\|_{L^2(\sqrt{\omega}\dd\omega)}^2 \lesssim_\mu D_\mu(g) + \mathcal{E}\text{rr}(M)   \|g\|_{L^2(\sqrt{\omega}\dd\omega)}^2
        \end{split}
    \end{equation}
    where $\mathcal{E}\text{rr}(M)\to 0$ as $M\to \infty$. 
In particular, if $$(g_n)_n \subset L^2(\sqrt{\omega})\ \text{ with }\ \| g_n \|_{L^2(\sqrt{\omega})}=1 \ \text{ and }\ D_\mu(g_n) \to 0, $$ then 
$$\lim_{M\to\infty}
    \limsup_{n\to\infty} \| \chi_{(M,\infty)}g_n \|_{L^2(\sqrt{\omega}\dd\omega)} =0. $$
\end{lemma}
\begin{proof}
    We only show \eqref{eq:lemm:transfer_coercivity1}, since the second claim follows immediately. 
    
    We are going to build a smooth cut off function, 
    $0\leq \eta_M \leq 1 \in C^\infty((0,\infty))$, for a rescaled $\eta \in C^\infty([0,\infty)))$ with $ 0 \leq \eta \leq 1$, $\eta(r)=1$ when $r\geq 1$ and  $\eta(r)=0$ when $r\leq 1/2$. 
    And so that the product 
    \begin{align}\label{eq:eq:transfer_coercivityProof0.5}
    \eta_M g \text{ is supported in } [M^*, \infty), 
    \end{align}
    for some $M^*>0$. Given that we have this function then 
    we consider the product $\eta_M h_\mu$ for which the collisions in the Dirichlet from become 
    \begin{equation}\label{eq:transfer_coercivityProof1}
        \begin{split}
         &-(\eta_M h_\mu)(\alpha\omega) -(\eta_M h_\mu)(\beta\omega) + (\eta_M h_\mu)(\gamma\omega)  + (\eta_M h_\mu)(\omega) = \\ &
         \hspace{0.8cm}
         \eta_M(\omega) \Big[ 
         -h_\mu(\alpha\omega) - h_\mu(\beta\omega) + h_\mu(\gamma\omega) +h_\mu(\omega) \Big]  +  \\ 
         & \hspace{1cm}  h_\mu(\beta\omega)
         \Big[\eta_M (\omega) - \eta_M(\beta\omega)
         \Big] +h_\mu(\gamma\omega)
         \Big[\eta_M (\gamma\omega) - \eta_M(\omega)
         \Big] + h_\mu(\alpha\omega)
         \Big[\eta_M (\omega) - \eta_M(\alpha\omega)
         \Big].
        \end{split}
    \end{equation}
    An application of \eqref{eq:SG_at_infty_2} for $\eta_M g$ yields
    \begin{equation} \label{eq:transfer_coercivityProof1.5}
        \begin{split}
            D_\mu(\eta_M g) \geq \lambda_\infty \|\eta_M g\|_{L^2(\sqrt{\omega})}^2
        \end{split}
    \end{equation}
    and using the computation above in \eqref{eq:transfer_coercivityProof1} together with the fact that $|y_1 + \dots + y_4|^2 \le 4 \sum_{j=1}^4|y_j|^2$, the left-hand side is 
    \begin{equation} \label{eq:transfer_coercivityProof2}
    \begin{split}
        &D_\mu^{restr}(\eta_M g) \lesssim \\
        &\int_0^\infty \left\{ \iint_{\mathcal{R}} \frac{\omega^{5/2}}{ \prod_{j\in \{1,\alpha, \beta, \gamma\}}(j \omega+\mu)} \eta_M^2 (\omega) \Big[ 
         -h_\mu(\alpha\omega) - h_\mu(\beta\omega) + h_\mu(\gamma\omega) +h_\mu(\omega) \Big]^2 \dd\alpha\dd\beta \right\} \dd\omega\\
         & + \sum_{\ell \in \{1,\alpha, \beta, \gamma\}} \int_0^\infty \iint_{\mathcal{R}} \frac{\omega^{5/2}}{ \prod_{j\in \{1,\alpha, \beta, \gamma\}}(j \omega+\mu)}|h_\mu (\ell\omega)|^2|\eta_M (\omega) - \eta_M(\ell\omega) |^2 \dd\alpha\dd\beta 
         \dd\omega \\
         & \lesssim D_\mu(g)  + \\
         & \sum_{\ell \in \{1,\alpha, \beta, \gamma\}} \int_0^\infty \iint_{\mathcal{R}} \frac{\omega^{5/2}}{ \prod_{j\in \{1,\alpha, \beta, \gamma\}}(j \omega+\mu)} | \ell\omega + \mu|^2 |g (\ell\omega)|^2
         |\eta_M (\omega) - \eta_M(\ell\omega) |^2 \dd\alpha\dd\beta 
         \dd\omega
        \end{split}
    \end{equation}
    where for the first line, we bounded the cutoff function to recreate the whole Dirichlet form. In the line below used  that $h_\mu(\ell\omega) = (\ell\omega + \mu)g(\ell\omega)$. 
    %Now for the second line we are going to use that  $\frac{\omega^{4}}{ \prod_{j\in \{1,\alpha, \beta, \gamma\}}(j \omega+\mu)}$ is bounded below and above (as $\mu>0$). This implies that the last line is equivalent to 
    Now for the last line we use that $1\leq \ell \leq C(\mathcal{R}):=\sup_{(\alpha, \beta) \in \mathcal{R} } 
  \max\{1,\alpha,\beta,\gamma\}<\infty$ to get that $\omega+\mu \leq (\ell\omega+\mu)\leq C_{\mathcal{R}} (\omega+\mu)$. Thus it is bounded by 
    \begin{equation} \label{eq:transfer_coercivityProof3}
        \begin{split} & 
        \lesssim \sum_{\ell \in \{1,\alpha, \beta, \gamma\}} \int_0^\infty \iint_{\mathcal{R}}  \frac{\omega^{5/2}}{(\omega+\mu)^2} 
    |g (\ell\omega)|^2|\eta_M (\omega) - \eta_M(\ell\omega) |^2 \dd\alpha\dd\beta \dd\omega\\
    &\lesssim \sum_{\ell \in \{1,\alpha, \beta, \gamma\}} \int_0^\infty \iint_{\mathcal{R}} \sqrt{\omega}|g (\ell\omega)|^2|\eta_M (\omega) - \eta_M(\ell\omega) |^2 \dd\alpha\dd\beta \dd\omega,  \end{split}
    \end{equation}
    since $\omega^2/(\omega+\mu)^2 \leq \mathcal{O}(1)$. 
    Finally if our cutoff function is built so that it additionally satisfies 
     \begin{align} \label{eq:transfer_coercivityProof3.5}
     \sup_{\omega>0} \omega |\eta_M'(\omega)| \lesssim \mathcal{E}\text{rr}(M), 
     \end{align}
     for some error term $\mathcal{E}\text{rr}(M)$, we may take the cutoff term outside of the integral as it will be bounded uniformly in $\omega$. Therefore under this additional condition, we may conclude that 
     \begin{align} \label{eq:transfer_coercivityProof4}
     \sum_{\ell \in \{1,\alpha, \beta, \gamma\}} \int_0^\infty \iint_{\mathcal{R}} \sqrt{\omega}|g (\ell\omega)|^2|\eta_M (\omega) - \eta_M(\ell\omega) |^2 \dd\alpha\dd\beta \dd\omega \lesssim \mathcal{E}\text{rr}(M)^2 \|g\|_{L^2(\sqrt{\omega})}^2.
     \end{align}
     after changing the variable $\ell\omega = \omega'$ and applying the mean value theorem: $| \eta_M (\omega) - \eta_M(\ell\omega)|\leq \int_{\omega}^{\ell\omega}| \eta_M'(s)| \dd s\leq \mathcal{E}\text{rr}(M) |\log (\ell)| \lesssim_{\mathcal{R}} \mathcal{E}\text{rr}(M)$ since $\ell$ is bounded above and below on $\mathcal{R}$.  
     
     Before we build this cutoff, let us show how we conclude: Combining \eqref{eq:transfer_coercivityProof1}, \eqref{eq:transfer_coercivityProof1.5}, \eqref{eq:transfer_coercivityProof2}, \eqref{eq:transfer_coercivityProof4} gives 
     $$ \lambda_\infty  \|\chi_{(M,\infty)}g\|_{L^2(\sqrt{\omega})}^2 \leq  \lambda_\infty \|\eta_M g\|_{L^2(\sqrt{\omega})}^2 \leq  D_\mu^{restr}(\eta_M g) \lesssim D_\mu(g) + \mathcal{E}\text{rr}(M)^2 \|g\|_{L^2(\sqrt{\omega})}^2$$
     which concludes the claim \eqref{eq:lemm:transfer_coercivity1}. 
     
     Finally, for the cutoff that does the job, take for $M\geq M_*$,  $$\eta_M(\omega) := \eta \left( \left(\frac{\omega}{M}\right)^{\epsilon_M} \right) $$
     which is $0$ whenever $\omega \leq M \left( \frac{1}{2} \right)^{1/\varepsilon_M}$ and is $1$ whenever $\omega\geq M$.  We then choose $\varepsilon_M$ so that we have precisely the desired support property \eqref{eq:eq:transfer_coercivityProof0.5}: $\left( \frac{1}{2} \right)^{1/\varepsilon_M} = M^*/M$, equivalently $ \varepsilon_M = \log2/\log(M/M^*)$, which tends to $0$ as $M \to \infty$, which in turn implies that $\mathcal{E}\text{rr}(M)$ from \eqref{eq:transfer_coercivityProof3.5} does as well.
     Before we finish let us comment that if one considers rather the standard cutoff with $\varepsilon_M=1$, then we would not get a uniform-in-$\omega$ bound! 
\end{proof}

\begin{lemma}[Compactness for the localised $K_\mu$]
\label{lem:local_K_mu_compact}
For every $0<M_1<M_2<\infty$, the operator
$$\chi_{[M_1,M_2]}K_\mu\chi_{[M_1,M_2]}:L^2(\sqrt{\omega}\dd\omega)\to L^2(\sqrt{\omega}\dd\omega)
$$
is compact.
\end{lemma}
\begin{proof}
This operator is in fact Hilbert-Schmidt and thus compact. Indeed by explicit calculations similar to \ref{lemm:Tightness_zero} show that the integral kernels of all three pieces of $K_\mu$, are bounded on the finite square $[M_1,M_2]^2$. Also in this square the norm on the flat $L^2$ and on the weighted $L^2(\sqrt{\omega})$ are equivalent. 
\end{proof}

We now have all the ingredients to give the proof of the Poincaré Inequality. 

\begin{proof}[Proof of Theorem \ref{theo:SGI}]
  We argue by contradiction. If the spectral gap inequality does not hold, then there exists a sequence $(g_n)_n$ with $g_n \perp \operatorname{Ker}(L_\mu)$, $\|g_n\|_{L^2(\sqrt{\omega}\dd\omega)}=1$ and $$ D_\mu(g_n) \to 0.$$
  Due to the boundedness of $(g_n)_n$ and weak compactness of the unit ball, we can find a weak limit, there exists $g$ so that $g_n \rightharpoonup g$. Now since $L_\mu$ is self-adjoint, non positive, and bounded we also have that $\|-L_\mu g_n\|_{L^2(\sqrt{\omega}\dd\omega)}^2 = \langle L_\mu^2 g_n,g_n \rangle \leq \|L_\mu\| \langle (-L_\mu g_n),g_n \rangle $ (since $-L_\mu \leq \|L_\mu\| I$ which is finite since $L_\mu$ is bounded). This gives that $\|-L_\mu g_n\|_{L^2(\sqrt{\omega}\dd\omega)}^2 \to 0$ and thus $L_\mu g_n \to 0$ strongly in $L^2(\sqrt{\omega})$. But also, $-L_\mu g_n \rightharpoonup -L_\mu g$ and so uniqueness of weak limits implies $L_\mu g=0$.  Thus $g \in \operatorname{Ker}(L_\mu)$ \footnote{
Equivalently for this argument, we can go through $\sqrt{-L_\mu}$ (since $-L_\mu$ is bounded, self-adjoint and nonnegative and so its square root is well-defined and bounded). Then $D_\mu(g_n)=\|\sqrt{-L_\mu} g_n\|_{L^2(\sqrt{\omega}\dd\omega)}^2\to 0.$
If $g_n\rightharpoonup g$ in ${L^2(\sqrt{\omega}\dd\omega)}$, then by continuity $\sqrt{-L_\mu} g_n \rightharpoonup \sqrt{-L_\mu} g$, and the left-hand side converges strongly to $0$. We conclude $\sqrt{-L_\mu} g=0$, and thus $g\in\ker L_\mu$.}. On the other hand, at the same time $g_n\perp \operatorname{Ker}(L_\mu)$, which is weakly closed (since the operator is closed) and so $g \perp \operatorname{Ker}(L_\mu)$. Therefore $g$ has to be $0$: $$g_n \rightharpoonup 0. $$ We will then show that $g_n \to 0$ strongly, by combining the local compactness with the coercivity in low and high frequencies, which will give us a contradiction. 

  From $\langle (-L_\mu g_n),g_n \rangle \to 0$ we have that $\langle A_\mu g_n, g_n\rangle - \langle K_\mu g_n, g_n\rangle \to 0$, and using Lemma \ref{lemm:A_mu} we get 
  \begin{equation}\label{eq:lowerbound on K_m}
  \frac{\pi^2}{6} \|g_n\|_{L^2(\sqrt{\omega})}^2 \leq \langle K_\mu g_n, g_n\rangle + o_n(1).
  \end{equation}
  
  If $K_\mu$ were compact we would be done. But it is not, so now the idea is to localise by introducing $\chi_{[M_1,M_2]}$ for $0<M_1<M_2<\infty$. Then 
  \begin{equation} \label{eq:SG_Theo_PRoof1}
      \begin{split}
          \langle K_\mu g_n, g_n\rangle & = \big\langle  K_\mu \chi_{[M_1,M_2]} g_n, \chi_{[M_1,M_2]} g_n \big\rangle 
          \\
          & \hspace{1cm}
          +  \big\langle  K_\mu (1-\chi_{[M_1,M_2]}) g_n, \chi_{[M_1,M_2]}g_n \big\rangle  
           \\ & \hspace{2cm}
          +\big\langle  K_\mu \chi_{[M_1,M_2]} g_n, (1-\chi_{[M_1,M_2]}) g_n \big\rangle 
          \\
          & \hspace{4cm} + 
          \big\langle K_\mu (1-\chi_{[M_1,M_2]}) g_n, (1-\chi_{[M_1,M_2]}) g_n\big\rangle.
      \end{split}
  \end{equation}
  The first term involves a compact operator, and so 
  $$\big\langle  K_\mu \chi_{[M_1,M_2]} g_n, \chi_{[M_1,M_2]} g_n \big\rangle = \langle \chi_{[M_1,M_2]} K_\mu \chi_{[M_1,M_2]} g_n,  g_n\rangle \to 0$$ since $\chi_{[M_1,M_2]} K_\mu \chi_{[M_1,M_2]}$ is compact, see Lemma \ref{lem:local_K_mu_compact}, and $g_n \rightharpoonup 0$.

For the rest of the terms we will use that $K_\mu$ is bounded and the tightness property at both infinity and zero from Propositions \ref{lemm:Tightness_zero} and \ref{lemm:transfer_coercivity} that give that no mass concentrates at neither $0$ nor infinity, i.e. 
  $$ \lim_{M_1 \downarrow 0}\limsup_{n\to \infty} 
\|\chi_{(0,M_1)}g_n\|_{L^2(\sqrt{\omega}d\omega)} =0
  $$ and 
 $$ \lim_{M_2 \uparrow \infty}\limsup_{n\to \infty} 
\|\chi_{(M_2, \infty)} g_n\|_{L^2(\sqrt{\omega}d\omega)} = 0, 
  $$ 
then we'd have that for all $\varepsilon >0$: we can find $M_1, M_2$ sufficiently small and large respectively so that  $\limsup_{n \to \infty} \| \chi_{(0,M_1)\cup(M_2, \infty)}g_n\||_{L^2(\sqrt{\omega}d\omega)}^2< \varepsilon$ or that  
  $$ \liminf_{n \to \infty} \| \chi_{[M_1,M_2]}g_n\|_{L^2(\sqrt{\omega}d\omega)}^2 \geq 1-\varepsilon,  
  $$
  (since the whole mass is $1$: $\|g_n\|_{L^2(\sqrt{\omega}d\omega)}=1$). This implies that for the mixed terms in \eqref{eq:SG_Theo_PRoof1}:
  \begin{align*}
    \limsup_{n \to \infty}  \vert \langle  K_\mu \chi_{[M_1,M_2]} g_n, (1-\chi_{[M_1,M_2]}) g_n\rangle \vert \leq \|K_\mu\|_{\mathcal{L}(L^2(\sqrt{\omega}))} \varepsilon
  \end{align*}
and 
\begin{align*}
    \limsup_{n \to \infty}  \vert \langle K_\mu (1-\chi_{[M_1,M_2]}) g_n, \chi_{[M_1,M_2]} g_n\rangle \vert \leq \|K_\mu\|_{\mathcal{L}(L^2(\sqrt{\omega}))} \varepsilon.
  \end{align*}
  Also regarding the last term, we write  
  \begin{align*}
    \limsup_{n \to \infty}  \vert \langle K_\mu (1-\chi_{[M_1,M_2]}) g_n, (1-\chi_{[M_1,M_2]}) g_n\rangle \vert \leq \|K_\mu\|_{\mathcal{L}(L^2(\sqrt{\omega}))} \varepsilon^2. 
  \end{align*}
  
  Then altogether we have that
  $$\limsup_{n\to\infty} \left\vert \langle K_\mu g_n , g_n\rangle_{L^2(\sqrt{\omega})} \right\vert \le 2\|K_\mu\|_{\mathcal{L}(L^2(\sqrt{\omega}))}\varepsilon + \|K_\mu\|_{\mathcal{L}(L^2(\sqrt{\omega}))}\varepsilon^2
  $$
  and since $\varepsilon>0$ was arbitrary, we conclude that 
  $$  \left\vert \langle K_\mu g_n , g_n\rangle_{L^2(\sqrt{\omega})} \right\vert \to 0. $$
  But at the same time we have from \eqref{eq:lowerbound on K_m} that 
  $$\frac{\pi^2}{6} \|g_n\|_{L^2(\sqrt{\omega})}^2 \leq \langle K_\mu g_n, g_n\rangle + o_n(1)$$
  and now the right-hand side tends to zero, which contradicts the mass one assumption. Thus a spectral gap has to exist, for every $g \perp \operatorname{Ker} L_\mu$. For arbitrary $g \in L^2(\sqrt{\omega}\dd\omega)$, we may decompose $g = \Pi g + (I-\Pi)g$, where $\Pi$ is the projection on the $\operatorname{Ker}(L_\mu)$ and use that 
  $D_\mu(g) = D_\mu((I-\Pi)g)$ (since $L_\mu$ is self-adjoint), to see the statement of the Theorem. 
\end{proof}

This implies exponential return to the Rayleigh-Jeans equilibrium.  

\begin{corollary}
Let $\Pi$ denote the $L^2(\sqrt{\omega}\dd\omega)$-orthogonal projection onto $\operatorname{Ker}L_\mu$. Then for the linear semigroup  
$S_\mu(t) = e^{tL_\mu}$ it holds that for all $t \geq 0$:
$$
    \|S(t)(I-\Pi)g\|_{L^2(\sqrt{\omega})}
    \le
    e^{-\lambda_\mu t}\|(I-\Pi)g\|_{L^2(\sqrt{\omega})}.
$$
\end{corollary}
\begin{proof}
    Standard energy estimate on $\| (I-\Pi) g_t\|_{L^2(\sqrt{\omega})}$ together with the Poincaré Inequality yields: 
    $$
    \frac{1}{2}\frac{\dd}{\dd t}\|S(t)(I-\Pi)g\|_{L^2(\sqrt{\omega})}^2
    =
    \big\langle L_\mu S(t)(I-\Pi)g,S(t)(I-\Pi)g \big\rangle_{L^2(\sqrt{\omega})}\le
    -\lambda_\mu\|S(t)(I-\Pi)g\|_{L^2(\sqrt{\omega})}^2.
$$
And then Grönwall implies the claim. 
%$$ \|S(t)(I-\Pi)g\| \le
 %   e^{-\lambda_\mu t}\|(I-\Pi)g\|.
%$$
\end{proof}

\section{Nonlinear Global Well-Posedness} \label{sec:GWP}
Given the existence of a spectral gap, we are now in shape to prove global in time well-posedness for the nonlinear problem, for small initial data. For this we are going to use the control of the nonlinearities in Section \ref{sec:LWP}.

\begin{theorem}
Let $\mu>0$ and $C_\mu$ the constant from Proposition \ref{prop: control NL}. For every $0<\delta<\lambda_\mu$, there exists $\rho = \rho(\mu,\delta)>0$ that satisfies 
$  \big( \rho+ \rho^2
    \big)
    \leq C_\mu^{-1} \big(\lambda_\mu-\delta\big)$, and is so that, if $g_0 \perp \operatorname{Ker}(L_\mu)$ with  $\|g_0\|_{L^2(\sqrt{\omega})} < \rho$, then the local in time solution found in Theorem \ref{Theo: LWP} is extended uniquely to a global solution $g \in C^1( [0,\infty);L^2(\sqrt{\omega}))$. 
    
    Moreover, if $g_0 \perp \operatorname{Ker}(L_\mu)$ then it remains so: $g_t \perp \operatorname{Ker}(L_\mu)$ for all $t\geq 0$ and
    \begin{equation}
        \|g_t\|_{L^2(\sqrt{\omega})} \leq e^{- \delta t} \|g_0\|_{L^2(\sqrt{\omega})} \quad \text{ for all } t\geq 0.
    \end{equation}
\end{theorem}
\begin{proof}
First we observe that the orthogonality to the $\operatorname{Ker}(L_\mu)$ is propagated by the
nonlinear flow. Indeed, if $\psi \in \operatorname{Ker}(L_\mu)$, then $\psi=n_\mu \phi$ where $\phi$ is collisional invariant (by characterization of the kernel in Prop. \ref{Prop:Kernel_L}). We compute 
\begin{align*}
\frac{d}{dt}\langle g_t,\psi\rangle_{L^2(\sqrt{\omega})}&=
\frac{d}{dt}\langle g_t,n_\mu\phi\rangle_{L^2(\sqrt{\omega})}=0
%\langle L_\mu g_t,\psi\rangle_{L^2(\sqrt{\omega})}+\langle\Gamma(g_t,g_t),\psi\rangle_{L^2(\sqrt{\omega})} + \langle Q(g_t,g_t, g_t),\psi\rangle_{L^2(\sqrt{\omega})} \\
%& = \langle\Gamma(g_t,g_t),\psi\rangle_{L^2(\sqrt{\omega})} + \langle Q(g_t,g_t, g_t),\psi\rangle_{L^2(\sqrt{\omega})}, 
\end{align*}
by conservation of the collision invariant $\phi$.

So, for as long as the solution $g_t$ exists, it remains orthogonal to the kernel, if it does so initially. 

Since $g_t\perp \operatorname{Ker} L_\mu$, the spectral gap and the nonlinear estimates from Proposition \ref{prop: control NL} give
\begin{equation}
\begin{split}
 \frac{1}{2}\frac{\dd}{\dd t}\|g_t\|_{L^2(\sqrt{\omega})}^2
&=\langle L_\mu g_t,g_t\rangle_{L^2(\sqrt{\omega})} +
\operatorname{Re}
\big\langle\Gamma (g_t,g_t),g_t \big\rangle_{L^2(\sqrt{\omega})} +
\operatorname{Re}
\big\langle Q (g_t,g_t,g_t),g_t \big\rangle_{L^2(\sqrt{\omega})} \\
 &\leq - \lambda_\mu \|g_t\|_{L^2(\sqrt{\omega})}^2 +
C_{\mu} \Big( \|g_t\|_{L^2(\sqrt{\omega})}^3+
 \|g_t\|_{L^2(\sqrt{\omega})}^4 \Big) \\
 &\leq
 \left( -\lambda_\mu+C_\mu\Big( \|g_t\|_{L^2(\sqrt{\omega})} +\|g_t\|_{L^2(\sqrt{\omega})} ^2 \Big)
\right)
\|g_t\|_{L^2(\sqrt{\omega})}^2.
\end{split}
\end{equation}
Now if $\|g_t\|_{L^2(\sqrt{\omega})} < \rho(\mu,\delta)$, with 
$C_\mu\big(\rho(\mu,\delta)+   \rho(\mu,\delta)^2 \big)\leq \lambda_\mu-\delta$, then 
\begin{align} \label{eq:GWP_proof1}
\frac{1}{2}\frac{\dd}{\dd t}\|g_t\|_{L^2(\sqrt{\omega})}^2 \leq -\delta  \|g_t\|_{L^2(\sqrt{\omega})}^2 \quad \Rightarrow\quad \|g_t\|_{L^2(\sqrt{\omega})}
\leq e^{-\delta t} \|g_0\|_{L^2(\sqrt{\omega})}.
\end{align}
Now it remains to show that $\|g_t\|_{L^2(\sqrt{\omega})} < \rho(\mu,\delta)$ for $0<t<T_{\max}$, where $T_{\max}$ is the maximal time of existence that we get from the local well-posedness, cf  Theorem \ref{Theo: LWP}. First note that by the same Theorem \ref{Theo: LWP} we have the blow-up criterion: 
$$ T_{\max}<\infty\quad \Rightarrow \quad \limsup_{t \uparrow T_{max}} \|g_t\|_{L^2(\sqrt{\omega})} = \infty, $$
which we will use to conclude that the solution is extended uniquely to a global one.

By continuity of $t\mapsto g_t$, we may assume that there exists a first time after $t=0$, say $t_*$, so that $\|g_{t_*}\|_{L^2(\sqrt{\omega})}=\rho(\mu,\delta)$. It has to be that $t_*>0$ by continuity and for all $\tau<t_*$, \eqref{eq:GWP_proof1} holds. In particular as $t \uparrow t_*$: 
$$ \|g_{t_*}\|_{L^2(\sqrt{\omega})} \leq e^{-\delta t_*}  \|g_{0}\|_{L^2(\sqrt{\omega})} < \rho(\mu,\delta), $$
which is a contradiction. Thus the solution $g_t$ has to remain bounded for as long as it exists: 
$$ \|g_{t}\|_{L^2(\sqrt{\omega})} < \rho(\mu,\delta),  \quad \forall\quad 0\leq t < T_{max},$$
or in other words $\sup_{0\leq t < T_{max}} \|g_{t}\|_{L^2(\sqrt{\omega})} < \rho(\mu,\delta)$.  Finally from the standard blow-up criterion above, we conclude that $T_{max}=\infty$. 

By Duhamel's formula the convergence rate is then improved up to $\lambda_\mu$.
\end{proof}

\bibliographystyle{alpha}
\bibliography{bibliography}

\end{document}